\documentclass[11pt]{amsart}

\usepackage{amsthm,amsfonts}
\usepackage{latexsym}
\usepackage{amssymb}
\usepackage{amsmath}
\usepackage{color}
\usepackage{bbm}
\usepackage{hyperref}
\usepackage{lmodern}
\usepackage{comment}

\allowdisplaybreaks 

\usepackage{graphicx}
\usepackage{tcolorbox}
\usepackage{tabularx}
\usepackage{tikz-cd}

\newtheorem{theorem}{Theorem}[section]
\newtheorem{lemma}[theorem]{Lemma}
\newtheorem{proposition}[theorem]{Proposition}
\newtheorem{corollary}[theorem]{Corollary}
\newtheorem{definition}[theorem]{Definition}

\theoremstyle{definition}

\newtheorem{example}[theorem]{Example}
\newtheorem{remark}[theorem]{Remark}

\newcommand{\cl}[1]{\mathcal{#1}}

\def\B{{\mathcal B}}

\def\B{{\mathcal B}}

\def\F{{\mathcal F}}

\def\L{{\mathcal L}}

\def\M{{\mathcal M}}

\begin{document}

\title{Fell bundles and Haagerup properties}

\author[E.B, R.C. L.T]{Erik B\'edos, Roberto Conti, Lyudmila Turowska}

\address{Department of Mathematics, University of Oslo,
P.B. 1053 Blindern, N-0316 Oslo, Norway}
\email{bedos@math.uio.no}
\address{Dipartimento SBAI, Sapienza Universit\`a di Roma,
Via A. Scarpa 16, I-00161 Roma, Italy}
\email{roberto.conti@sbai.uniroma1.it}
\address{Department of Mathematical Sciences, Chalmers University of Technology and the University of Gothenburg, Gothenburg SE-412 96, Sweden}
\email{turowska@chalmers.se}

\begin{abstract} We introduce new properties of Haagerup-type for Fell bundles over discrete groups and discuss their relationship with Haagerup properties for the 
associated $C^*$-algebras.  In particular, our results on Fell bundles associated with twisted unital $C^*$-dynamical systems extend  previous results known in the untwisted case. 
\end{abstract}

\date{\today}

\maketitle

\tableofcontents

\section{Introduction}

Fell bundles are a kind of device designed to provide several useful tools to analyze in efficient way $C^*$-algebras and the maps between them, often in connection with the existence of 
suitable gradings. As the name suggests, these are actually bundles of Banach spaces over groups (or more general objects like groupoids), with additional operations with a clear $C^*$-flavour. In this paper, for simplicity, we will stick to Fell bundles over discrete groups.

Over the recent years, there has been a consistent trend for transferring analytic properties from groups 
to the realm of Fell bundles. Not surprisingly, this has been 
investigated in great detail for amenability, e.g.~ in \cite{exel97, takeishi, exel,  abf,  bew24, bc25, bf25,bf25b},
including very recently polynomial growth \cite{fl25},
and for rapid decay in \cite{bk26}.
Since the cross-sectional $C^*$-algebras of Fell bundles are natural generalizations of $C^*$-crossed products associated to twisted $C^*$-dynamical systems, and many $C^*$-algebras can be described in this way, this extension is a very reasonable and often useful step.
However, a subtle issue is that there are often competing ways to 
transfer these notions, making it challenging to verify their equivalence or clarify their exact relationships.
Of course, as several facets are involved, some technicalities unavoidably show up in this process.

With these considerations in mind, we decided to start a long-term project devoted to studying how various group approximation/representation theoretic properties can be extended to Fell bundles.
In the present paper, we focus on the Haagerup property. In addition to considering Haagerup properties for groups and $C^*$-algebras,
we introduce some new  properties of Haagerup-type for Fell bundles and examine the possible interconnections from multiple angles. 
Many of our results generalize results for 
$C^*$-dynamical systems and their crossed products available in the literature (e.g.~\cite{bc12, mstt, mw21, gm22, kls}). 

There are many good reasons why the Haagerup property for groups \cite{ccjjv, bo} (some-times also referred to as a-T-menability) has been extensively studied.
For example, it plays a pivotal role in the study of the Baum-Connes/Novikov conjectures (cf.~\cite[Subsection 1.3.3]{ccjjv} for a short review), and is an important ingredient in Popa's deformation/rigidity theory  (cf.~\cite{adpo, hou}). Its origin \cite{choda} is deeply rooted in the theory of $II_1$ factors, as being a natural candidate for relaxing amenability/hyperfiniteness.
Formulating Haagerup-type properties for Fell bundles relies to some extent on natural requests that one way or another imitate the corresponding group and dynamical theoretic features,
but at the same time should be validated by the strength of the resulting theorems.

Since the early work of Haagerup \cite{haa1,haa2}, a recurrent idea 
has been to characterize various properties of groups via the existence of certain nets of positive definite  functions. This perspective inspired the formulation of analogous approximation properties for $C^*$-algebraic objects, such as $C^*$-algebras and $C^*$-dynamical systems, that often related back to the group properties one started with.
In this $C^*$-setting, the nets of positive definite functions are replaced by nets of operator-valued functions satisfying suitable forms of positive definiteness.

In this paper, we discuss several properties for Fell bundles having the flavor of the Haagerup property for groups or the Haagerup property for $C^*$-algebras equipped with a state.
The concept of Haagerup property for $(A, \psi)$ when $A$ is a $C^*$-algebra and $\psi$ is a state on $A$ has been considered in several papers (originally, only defined for faithful tracial states), see e.g.~\cite{dong, suzuki13, mw21}. Further, if $\Sigma=(A, G, \alpha)$ is a unital discrete $C^*$-dynamical system, 
a Haagerup property for $\Sigma$ (in the sense of Dong and Ruan) was introduced in \cite{dong_ruan}. Moreover, when $\psi$ is an $\alpha$-invariant state on $A$, a Haagerup property for $(\Sigma, \psi)$ was discussed in \cite{gm22} (following \cite{mstt} when $\psi$ is faithful and tracial). 
To each twisted unital discrete $C^*$-dynamical system $\Sigma=(A, G, \alpha, \sigma)$ one may associate a Fell bundle $\B_\Sigma$ such that $C_r^*(\B_\Sigma)$ is $*$-isomorphic to the reduced $C^*$-crossed product $C_r^*(\Sigma)$. Such Fell bundles may be characterized as so-called crossed product bundles, see Subsection \ref{cpb}
for details. We will often find it convenient to formulate results about such twisted systems in terms of crossed product bundles. 

Let $\B=(B_g)_{g\in G}$ be a Fell bundle over $G$ \cite{exel} and $\psi$ be a state on the unit fiber $B_e$. We define in Section \ref{sec:haag_bundle}
the Haagerup property for the pair $(\B, \psi)$ in terms of the existence of a net of positive definite $\B$-bundle maps (as defined in \cite{bc25}) satisfying suitable properties, cf.~Definition \ref{def:haag_bundle}. If $C_r^*(\B)$ denotes the associated reduced cross-sectional $C^*$-algebra and  
$E_e:C_r^*(\B)\to B_e$ is the canonical conditional expectation, we can then talk about the Haagerup property for $(B_e, \psi)$ and the one for $(C_r^*(\B), \psi\circ E_e)$.
As we show in Theorem \ref{HP-Fell}, the Haagerup property for $(C_r^*(\B), \psi\circ E_e)$ turns out to be equivalent to the Haagerup property for  $(\B,\psi)$, which implies that $(B_e, \psi)$ has the Haagerup property. 
On the other hand, the  PD-approximation property for $\B$ was defined in \cite{bc25}. As shown recently in \cite{bf25}, it is equivalent to Exel's approximation property \cite{exel97, exel} (at least when $\B$ is unital or $B_e$ is nuclear).  As this property is reminiscent of amenability for groups, it inspired us to introduce 
a Haagerup PD-approximation property for $\B$, cf.~Definition \ref{HPD}, generalizing the Haagerup property (in the sense of \cite{dong_ruan}) for unital $C^*$-dynamical systems.  
When $A\subset C$ is a unital inclusion of $C^*$-algebras having a faithful conditional expectation $E:C\to A$, a Hilbert $A$-module Haagerup property for $C$ w.r.t.~$E$ 
is discussed in \cite{dong_ruan}. Thus, when $B_e$ is unital, it makes sense to consider the 
Hilbert $B_e$-module Haagerup property for $C_r^*(\B)$ w.r.t.~$E$. We show in Theorem \ref{moduleHaag} that if this property is satisfied, then $\B$ has the Haagerup PD-approximation property, while the converse result is shown to be true under certain additional assumptions in Theorem \ref{orthoH}.  
Finally, it seemed also natural to us to generalize the notion of nuclearity for $C^*$-dynamical systems introduced in \cite{mstt} to the setting of Fell bundles. We 
show in Theorem \ref{nuclearC*} that it is equivalent to $\B$ having the PD-approximation property and $B_e$ being nuclear, which is known to be equivalent to the nuclearity of
$C_r^*(\B)$.

With so many properties at hand our main concern in this note has been to examine in detail the way they are logically related. Some of the implications we found hold for any Fell bundle, others for unital ones and some only for crossed product bundles. Even so, some of our results extend known facts for crossed product to the setting of twisted dynamical systems. For the benefit of the reader, we have included a diagram in Section \ref{diagram} illustrating the connections between the properties considered in this paper.  

In forthcoming publications we plan to give a closer look to suitable analogs of Herz-Schur multipliers and cover some other properties for Fell bundles,
like weak amenability, weak Haagerup property and Kazhdan's property (T).  

\vspace{-1.9ex}
\section{Preliminaries}
  
\subsection{Fell bundles over discrete groups}
Throughout this paper, $\B=(B_g)_{g\in G}$ will denote a Fell bundle over a discrete group $G$, and 
$e$ will denote the unit of $G$. We recall that a \emph{Fell bundle $\B=(B_g)_{g\in G}$} over a  $G$ is a collection of Banach spaces %(called \emph{fibers}) 
satisfying the following properties (where $\B$ also denotes the disjoint union of the $B_g$'s).
There is an associative multiplication map from $\B\times \B$ into $\B$ 
such that $B_g B_h \subseteq B_{gh}$, $(a, b)\mapsto ab$ is bilinear on $B_g\times B_h$ for all $g, h \in G$, and $\|ab\| \leq \|a\|\|b\|$ for all $a, b\in \B$. Moreover, there is an involutive, anti-multiplicative map $b\mapsto b^*$ from $\B$ into itself
such that $B_g^*= B_{g^{-1}}$. Moreover, $b\mapsto b^*$ is a conjugate-linear, norm-preserving map from $B_g$ into $B_{g^{-1}}$ for all $g \in G$.  Finally, we have $\|b^*b\| =\|b\|^2$ and $b^*b \geq 0$ in the unit fibre $B_e$ (which is a $C^*$-algebra) for all $b\in \B$. Note that each $B_g$ is a right $B_e$-module (with $\langle b, c \rangle_{B_e} = b^*c$ for $b, c\in B_g$) and also a left $B_e$-module.  We say that $\B$ is \emph{unital} when $B_e$ is unital.

Our notation will essentially be as in Exel's book \cite{exel}. Thus, a section of $\B$ is a function $f:G\to \B$ such that $ f(g) \in B_g \text{ for every } g \in G$,
and $C_c(\B)$ denotes the $*$-algebra consisting of all finitely supported sections of $\B$, the product and involution being given by
\[ (f_1 \star f_2) (h) = \sum_{g\in G} f_1(g) f_2(g^{-1}h), \quad f^*(h) = f(h^{-1})^*\]
for all $h \in G$. Also, $\ell^2(\B) $ denotes the (right) Hilbert $B_e$-module obtained as the completion of $C_c(\B)$, considered as
 a (right) inner product  $B_e$-module  with operations given by 
 \[(\xi \cdot b)(g) = \xi(g) b, \quad  \langle \xi, \eta\rangle_{B_e} = \sum_{g\in G}\,\xi(g)^*\eta(g)\quad \text{for every } g\in G,\]
$\xi, \eta \in C_c(\B)$ and $b \in B_e$.
 The (left) \emph{regular representation} $\lambda^\B = (\lambda^\B_g)_{g\in G}$ of $\B$ in the $C^*$-algebra of adjointable operators on $\ell^2(\B)$, which we will denote by $\L_{B_e}(\ell^2(\B))$,  is determined  for each $g\in G$ by
\[ (\lambda^\B_g(b)\xi)(h) = b\,\xi(g^{-1}h) \quad \text{for all } b \in B_g, \xi \in C_c(\B) \text{ and } h \in G.\]
The map $\lambda^\B_g:B_g\to \L_{B_e}(\ell^2(\B))$ is isometric for each $g\in G$. Further, 
$\lambda^\B$ induces a faithful $*$-representation $\iota^\B : C_c(\B)\to \L_{B_e}(\ell^2(\B))$ given  by
\[\iota^\B(f) = \sum_{g\in G} \lambda^\B_g(f(g)) \quad \text{ for all } f\in C_c(\B),\]
see  \cite[Proposition 17.9, (ii)]{exel}.
 
 The \emph{reduced cross-sectional $C^*$-algebra $C_r^*(\B)$} associated to $\B$ is by definition the $C^*$-subalgebra of $\L_{B_e}(\ell^2(\B))$ generated by 
 $\{ \lambda_g^\B(b): g\in G, b \in B_g\}$, i.e., by $\iota^\B(C_c(\B))$. Alternatively, setting $\|f\|_r := \|\iota^\B(f)\|$ for $f\in C_c(\B)$,  $C_r^*(\B)$ may be considered as the completion of $C_c(\B)$ w.r.t.~the norm $\|\cdot\|_r$. 
For each $g\in G$ there is a contractive linear map $E_g:C_r^*(\B) \to B_g$ satisfying
\[E_g\big(\lambda^\B_h (b)\big)=\begin{cases} b, \quad \text{ if } g =h, \\  0, \quad \text{ if }g\neq h, \end{cases}\]
cf.~\cite[Lemma 17.8]{exel}. For $x \in C_r^*(\B)$, $E_g(x)$ may be thought of as the Fourier coefficient of $x$ at $g$. The map $E_e$ is a faithful conditional expectation from $C_r^*(\B)$ onto $B_e$ (after identifying $B_e$ with $\lambda_e^\B (B_e)$, as we will do in the sequel), see \cite[Propositions 17.13 and 19.3]{exel}.

We will make use of the matrix $C^*$-algebras associated to $\B$ as in \cite{af19}. We recall their definition. Let $g_1, \ldots, g_n \in G$ and set $\mathbf{g}:= (g_1, \ldots, g_n) \in G^n$.Then
\[M_\mathbf{g}(\B) := \big\{R=[r_{ij}]\in M_n(\B) : r_{ij} \in B_{g_i^{-1}g_j} \text{ for all } i, j =1, \ldots, n\big\}\]
is a $*$-algebra with respect to the natural operations, which can be equipped with a norm turning it into a $C^*$-algebra, cf.~\cite[Lemma 2.8]{af19}.

We next recall from \cite{bc25} that a \emph{$\B$-bundle map} is a family of maps $T = (T_g)_{g\in G}$  such that $T_g: B_g \to B_{g}$ is linear and bounded for every $g\in G$. Every $\B$-bundle map $T$ induces a linear map $\phi_T: C_c(\B) \to C_c(\B)$ given by 
\[[\phi_T(f)](g)= T_g(f(g))\quad \text{for all } f\in C_c(\B) \, \text{and} \, g\in G.\]
A $\B$-bundle map $T$ is said to be \emph{reduced} when $\phi_T$ is bounded w.r.t.~$\|\cdot\|_r$, in which case we denote by $M_T$ the extension of $\phi_T$ to a bounded linear operator on $C_r^*(\B)$. As a bounded linear operator, $M_T$ is then uniquely determined by
\begin{equation} \label{M_T}
M_T(\lambda^\B_g(b)) = \lambda^\B_g(T_g(b)) \quad \text{ for all } g\in G \text{ and } b\in B_g.
\end{equation}
 Moreover, a $\B$-bundle map $T$ is said to be \emph{positive definite} \cite{bc25} when
\begin{equation}\label{Bposdef}
\sum_{i, j=1}^n b_i \, T_{g_i^{-1}g_j} (a_i^* a_j) \, b_j^* \, \geq \, 0 \, \, (\text{ in } B_e)
\end{equation}
for all $n\in \mathbb{N}$, $g_1, \ldots, g_n \in G$ and  $a_i, b_i \in B_{g_i}$, $i=1, \ldots, n$.
Equivalently, cf.~\cite[Remark 3.4]{bc25}, $T$ is positive definite when the matrix 
\[\Big[ T_{g_i^{-1}g_j} (a_i^* a_j)\Big]\] is positive in $M_{\mathbf{g}}(\B)$  for all  $n\in \mathbb{N}$, $\mathbf{g}= (g_1, \ldots, g_n) \in G^n$ and $a_i \in B_{g_i}$, $i=1, \ldots, n$.  
The following result is a special case of \cite[Theorem 3.14]{bc25}: 
\begin{theorem} \label{PDCP}
Let $T$ be a $\B$-bundle map. Then  $T$ is positive definite
if and only if $T$ is reduced with $M_T: C^*_r(\B) \to C^*_r(\B)$ being completely positive.
\end{theorem}

\subsection{Crossed product bundles} \label{cpb}
We will find it convenient to formulate many of our results dealing with unital discrete twisted  $C^*$-dynamical systems and their associated reduced $C^*$-crossed products
in the setting of so-called crossed product bundles, as introduced in \cite{mawa}.
\begin{definition}
    A unital Fell bundle ${\mathcal B}=(B_g)_{g\in G}$ is said to be a \emph{(unital, twisted) crossed product bundle} if there exists a section $u: G\to \B$ which is unitary, i.e., satisfies $u(g)u(g)^*= u(g)^*u(g) = 1_{B_e}$ for all $g\in G$.
\end{definition}
As shown in \cite[Lemma 5.2]{mawa}, we then have $B_g=\{au(g): a \in B_e\}$ and  $B_{gh}=B_gB_h$ for all $g, h\in G$. For brevity we will sometimes write $u_g$ instead of $u(g)$ for $g\in G$.

Let $\Sigma=(A, G, \alpha, \sigma)$ be a discrete unital twisted $C^*$-dynamical system, see for example \cite{zm, pr, bc12}. That is, $\alpha:G\to {\rm Aut}(A)$ is a map from $G$ to the group of $*$-automorphisms of a unital $C^*$-algebra $A$ and $\sigma:G\times G \to \mathcal{U}(A)$ is a map from $G\times G$ to the unitary group of
 $A$ satisfying
 \begin{eqnarray*}
    &\alpha_g\alpha_h=\text{Ad}(\sigma(g,h))\alpha_{gh},\\
    &\sigma(g,h)\sigma(gh,k)=\alpha_g(\sigma(h,k))\sigma(g,hk),\\
    &\sigma(g,e)=\sigma(e,g)=1_{A}
\end{eqnarray*}
for all $g, h, k\in G$ (where ${\rm Ad}(v)$ denotes the $*$-automorphism of $A$ given by $[{\rm Ad}(v)](a)= vav^*$ for all $a\in A$). One also says that $(\alpha, \sigma)$ is a \emph{twisted action of $G$ on $A$}. It is well-known and easy to check that the associated Fell bundle $\B_\Sigma =(A\times \{g\})_{g\in G}$ with multiplication and involution given by
 \[(a, g)(b, h) := \big(a\alpha_g(b)\sigma(g,h), gh\big), \,\, (a, g)^* := \big(\alpha_g^{-1}(a^*)\sigma(t^{-1}, t)^*, g^{-1}\big)\]
 for all $a,b\in A$ and $g, h\in G$, is a crossed product bundle (with $u(g):=(1_A, g)$ for all $g\in G$), such that $C_r^*(\B_\Sigma)\simeq C_r^*(\Sigma)$ (the reduced $C^*$-crossed product associated to $\Sigma$).

Conversely, any crossed product bundle $\B$ is isomorphic to $\B_\Sigma$ for some discrete unital twisted $C^*$-dynamical system $\Sigma$. To see this, assume ${\mathcal B}=(B_g)_{g\in G}$ is a crossed product bundle and fix a unitary section $u:G\to \B$. We may clearly assume that $u(e)=1_{B_e}$. Define $\alpha_g(b):=u(g)bu(g)^* \in B_e$ for all $b\in B_e$ and $g\in G$, and  $\sigma(g,h):= u(g)u(h)u(gh)^* \in B_e$ for all $g$, $h\in G$. Then $(\alpha, \sigma)$ is a twisted action of $G$ on $B_e$. 
Indeed, for $b\in B_e$, $g$, $h$, $k\in G$, we have 
\begin{eqnarray*}
    (\alpha_g\alpha_h)(b)&=&u(g)u(h)bu(h)^*u(g)^*\\&=&\sigma(g,h)u(gh)bu(gh)^*\sigma(g,h)^*\\&=&
    (\text{Ad}(\sigma(g,h))\alpha_{gh})(b);
    \end{eqnarray*}
    and 
    \begin{eqnarray*}
\sigma(g,h)\sigma(gh,k)&=&u(g)u(h)u(gh)^*u(gh)u(k)u(ghk)^*\\&=&u(g)u(h)u(k)u(ghk)^*\\&=&u(g)\sigma(h,k)u(g)^*u(g)u(hk)u(ghk)^*\\
&=&\alpha_g(\sigma(h,k))\sigma(g,hk).
\end{eqnarray*}
Set $\Sigma:= (B_e, G, \alpha, \sigma)$.  Then it is routine to check that the bundle map $S=(S_g)_{g\in G}$ from $\B$ to $\B_\Sigma$ defined by $S_g(b):= (bu(g)^*, g)$ for $g\in G$ and $b \in B_g$ is an isomorphism, with inverse $S^{-1}=(S^{-1}_g)_{g\in G}$ given by $S_g^{-1}(a, g)= au(g)$ for all $a\in B_e$ and $g\in G$. It now follows from \cite[Section 21.1]{exel} that 
\[\, C_r^*(\B)\simeq C_r^*(\B_\Sigma) \simeq C_r^*(\Sigma).\] 
 For completeness, we mention that one may also associate a Fell bundle to any twisted \emph{partial} action of $G$ on a $C^*$-algebra $A$, as defined in \cite{exel97a}. 

\subsection{Approximation properties for Fell bundles}
We recall the following definition, cf.~\cite[Definition 20.4]{exel}.
\begin{definition}\label{ap}
 A Fell bundle $\mathcal B$ over $G$ is said to have \emph{Exel's approximation property (AP)} if there exists a net $\{\xi_i\}_{i\in I}$ of finitely supported functions from $G$ to $B_e$ such that
 \begin{itemize}
     \item 
     $\sup_{i\in I}\|\sum_{h\in G}\xi_i(h)^*\xi_i(h)\|
     < \infty$,
     
     \smallskip
     \item $\lim_i\|\sum_{h\in G}\xi_i(gh)^*b\xi_i(h)-b\|=0$ for every $g\in G$ and $b\in B_g$.
 \end{itemize}
\end{definition}
Exel's AP is satisfied whenever $G$ is amenable, cf.~\cite[Theorem 20.4]{exel}. In \cite[Theorem 3.2]{bf25}, Buss and Ferraro provide a list of statements all equivalent to Exel's 
AP, and say that $\B$ is  \emph{$C^*$-amenable} when one of these conditions holds. We will adopt this terminology in the sequel.

We also recall the following definition from \cite{bc25}. Here, a $\B$-bundle map $T=(T_g)_{g\in G}$ is said to have finite support when its support $\{g\in G: T_g\neq 0\}$ is finite. 
\begin{definition} \label{PD} A Fell bundle $\B$ over $G$ is said to have the \emph{PD-approximation property} if there exists a net $\{T^i\}_{i\in I}$ of $\B$-bundle maps satisfying the following properties:
 \begin{itemize}
 \item[(i)] For each $i \in I$, $T^i$ is positive definite  and has finite support; 
  \item[(ii)]  $\{T^i\}_{i\in I}$ is uniformly bounded in the sense that $\sup_{i\in I}\|T^i_e\| < \infty$;
  \item[(iii)] $\lim_i \|T^i_g(b) - b\| = 0$ for every $g \in G$ and $b\in B_g$. 
 \end{itemize} 
 \end{definition}
Note that  $\mathcal B$ is said to have the BCAP in \cite{bf25} when $\B$ has the PD-approximation property, but satisfies $\sup_{i\in I}\|T^i_e\| \leq 1$.
If $\B$ is unital, these two notions are easily seen to be equivalent. 
Moreover, Buss and Ferraro show that $C^*$-amenability of $\B$ is equivalent to the BCAP for $\B$, cf.~\cite[Corollary 4.17]{bf25}.    

\section{Haagerup properties for groups, C*-algebras and C*-dynamical systems} 

\subsection{Haagerup property for discrete groups} A discrete group $G$ is said to have the \emph{Haagerup property} \cite{bo} if there exists a net $\{\varphi_i\}_{i\in I}$ of normalized positive definite functions on $G$ vanishing at infinity such that $\varphi_i \to 1$ pointwise on $G$. See \cite{ccjjv} for some equivalent definitions when $G$ is countable.

\subsection{Haagerup property for $C^*$-algebras w.r.t.~ states} Let $A$ be a $C^*$-algebra $A$ and $\psi$ be a state on $A$.
We consider the GNS-construction associated to $\psi$ and denote by $\langle \cdot, \cdot\rangle_{\psi}$ the (right) semi-inner product on $A$ defined by 
\[ \langle a, b\rangle_{\psi} = \psi(a^*b) \quad \text{for all} \, a,b \in A.\]
Setting $N_{\psi}:=\{ a \in A : \psi(a^*a)=0\}$ and $A_{\psi}:= A/N_{\psi}$, we get an inner product on $A_{\psi}$ given by
\[\langle a + N_{\psi}, b + N_{\psi}\rangle_{\psi} = \psi(a^*b) \quad \text{for all} \, a, b \in A.\]
In the sequel, we will sometimes write $[a]_{\psi}$ instead of $a + N_{\psi}$ for elements in $A_{\psi}$.
We let $H_{\psi}$ denote the Hilbert space obtained from completing $A_{\psi}$, and 
write $\|\cdot\|_\psi$ for the associated norm.
For $a\in A$, we also set \[\|a\|_{\psi}:= \|[a]_\psi\|_{\psi}= \psi(a^*a)^{1/2} \leq \|a\|.\] 

Now, if $\Phi:A\to A$ is a contractive completely positive map satisfying $\psi\circ \Phi \leq \psi$, we get a contractive linear map $S_\Phi \in \B(H_\psi)$ determined by 
\[S_\Phi([a]_\psi) = [\Phi(a)]_\psi \quad \text{ for all } a \in A,\]
cf.~[Lemma 3.1]\cite{mw21}. According to \cite[Definition 3.2]{mw21}, the pair \emph{$(A, \psi)$ is said to have the Haagerup property} whenever there exists a net $\{\Phi_i\}_{i\in I}$ of contractive completely positive maps on $A$ satisfying
\begin{itemize}
    \item[a)] $\psi\circ \Phi_i \leq \psi$ for every $i\in I$;
    \item[b)] each $S_{\Phi_i}$ is a compact operator on $H_\psi$;
    \item[c)] $\{S_{\Phi_i}\}_{i\in I}$ converges to $I_{H_\psi}$ in the strong operator topology (SOT).
\end{itemize}
Note that  $c)$ is equivalent to requiring that $\lim_i \|\Phi_i(a)-a\|_{\psi} = 0$ for all $a\in A$. 
This definition, which is even formulated for proper weights in \cite{mw21}, generalizes Dong's definition \cite{dong}
 when $\psi$ is assumed to be a faithful tracial state on $A$ (as for instance in \cite{suzuki13}). 
 When $A$ is unital,  one may always assume that each $\Phi_i$ is unital and $\psi\circ \Phi_i = \psi$ for every $i\in I$, cf.~\cite[Remark 3.4]{mw21}. 

 For a discrete group $G$, it is well known that $G$ has the Haagerup property if and only if $(C_r^*(G), \tau)$ has the Haagerup property, where $\tau$ denotes the canonical tracial state on $C_r^*(G)$, cf.~\cite[Theorem 2.6]{dong} and  \cite[Theorem 2.4(1)]{suzuki13}.
 If $A$ is a  unital nuclear $C^*$-algebra and $\tau$ is a  faithful tracial state on $A$, then $(A,\tau)$ has the Haagerup property \cite[Theorem 3.6]{suzuki13}.
 
\begin{example} \label{CAR1}
As we are not aware of any specific place in the literature with a discussion of examples of a unital $C^*$-algebra $A$ having a faithful non-tracial state $\psi$ such that $(A,\psi)$ has the Haagerup property, we mention here a notable one. Let $M_{2^\infty}$ be the CAR-algebra (i.e., the UHF algebra of type $2^\infty$), and let $\omega_\lambda$ be the Powers state on $M_{2^\infty}$ associated to some $\lambda\in (0,1/2)$, cf.~\cite[Definition 4.2]{powers}. Then $\omega_\lambda$ is faithful, but not tracial, and it is not difficult to show that $(M_{2^\infty}, \omega_\lambda)$ has the Haagerup property. More generally, similar examples can be produced for any UHF-algebra. Of course, UHF-algebras are nuclear, and one may wonder whether Suzuki's result mentioned above also holds for a unital nuclear $C^*$-algebra $A$ when $\tau$ is a non-tracial faithful state on $A$.  
\end{example}

\subsection{Haagerup properties for discrete unital $C^*$-dynamical systems} Let $\Sigma=(A, G, \alpha)$ is a discrete unital $C^*$-dynamical system. 
Recall from \cite{Claire} that  a map $\varphi: G \to A$ is called \emph{$\alpha$-positive definite} when \[[\alpha_{g_i}(\varphi(g_i^{-1} g_j))] \in M_n(A)^+\] for all $n \in {\mathbb N}$ and $g_1,\ldots,g_n \in G$. Also, a function $f: G \to A$ is said to vanish at infinity if, for any given $\epsilon >0$, there exists a finite subset $F \subset G$ such that $\|f(g)\| < \epsilon$
for every $g \in G \setminus F$, i.e., the function $g\mapsto \|f(g)\|$ belongs to $C_0(G)$. We will denote by $C_0(G,A)$ the space of all functions from $G$ to $A$ vanishing at infinity.

Recall that a function $f: G \to A$ is said to be normalized when $f(e) = 1_A$. The following definition is based on \cite{dong_ruan},  
but slightly adapted to fit with the convention used in \cite{Claire}. (See also \cite[Definition 5.6]{kls} for a more general definition in the case of twisted groupoid $C^*$-dynamical systems). 
If $A$ is a $C^*$-algebra, then we denote its center by $Z(A)$.

 The system \emph{$\Sigma = (A,G,\alpha)$ is said to have the Haagerup property (in the sense of Dong and Ruan)} if there exists a net $\{\varphi_i\}_{i\in I}$ of normalized $\alpha$-positive definite functions in $C_0(G,Z(A))$ such that \[\lim_{i} \|\varphi_i(g)-1_A\| = 0\] for every $g\in G$.
 It is immediate that $\Sigma$ has the above Haagerup property whenever $G$ has the Haagerup property.

 \medskip 
 Assume now that  there exists an  $\alpha$-invariant state $\tau$ on $A$. There is a notion of Haagerup property for $(\Sigma, \tau)$ introduced in \cite{gm22}, generalizing the one given in \cite{mstt} when $\tau$ is also assumed to be faithful and tracial. Recall that if $F=(F_g)_{g\in G}$ is a positive definite $\Sigma$-multiplier (equivalently, $T= (T_g)_{g\in G}$, where for every $g\in G$, $T_g((a, g)):=(F_g(a), g)$ for all $a\in A$, is a positive definite $\B_\Sigma$-bundle map)  
 such that $F_e$ is unital and $\tau\circ F_e \leq \tau$, then for each $g\in G$, the map $[a]_\tau \mapsto [F_g(a)]_\tau $
is bounded w.r.t.~$\|\cdot\|_{\tau}$, hence extends to a bounded operator $\widetilde F_g$ on $H_\tau$, cf.~\cite[Lemma 2.4]{gm22} and \cite[Lemma 3.2]{mstt}. The definition is as follows.

 \begin{definition} The pair \emph{$(\Sigma, \tau)$ is said to have the Haagerup property} if there exists a net $\{F^i\}_{i\in I}$ of positive definite $\Sigma$-multipliers such that 
\begin{itemize}
\item[i)] each $F^i_e$ is unital and $\tau\circ F^i_e \leq \tau$ for every $i\in I$;
\item[ii)] $\widetilde F^i_g$ is a compact operator on $H_\tau$ for every $i\in I$ and $g\in G$;
\item[iii)] the map $g\mapsto \|\widetilde F^i_g\|$ vanishes at infinity for each $i\in I$;
\item[iv)] $\lim_i \|F^i_g(a) - a\|_{\tau}= 0$ for every $g\in G$ and $a\in A$.
\end{itemize}
\end{definition}
In this definition,  $\tau$ is required to be $\alpha$-invariant, but it makes sense without this requirement. In the next section we will propose a definition of Haagerup property for Fell bundles and states without imposing such a restriction (which in that general setting has no obvious meaning).
We gather below some known results.
\begin{theorem} \label{HP-dyn}
Let $E:C_r^*(\Sigma)\to A$ denote the canonical conditional expectation. Then  $(\Sigma, \tau)$ has the Haagerup property 
if and only if  $(C_r^*(\Sigma), \tau\circ E)$ has the Haagerup property, 
in which case $G$ has the Haagerup property and $(A, \tau)$ has the Haagerup property.
\end{theorem}
\begin{proof}
See \cite[Theorem 3.8, Corollary 3.9]{mstt}, \cite[Theorem 2.9]{gm22} and \cite[Corollary 3.14 and Proposition 3.16]{mw21}. 
\end{proof}
It follows immediately from this result that $\Sigma$ has the Haagerup property (in the sense of Dong-Ruan) whenever $(\Sigma, \tau)$ has the Haagerup property for some invariant state $\tau$ on $A$. Note also that it is not necessarily true that $(C_r^*(\Sigma), \tau\circ E)$ has the Haagerup property when $G$ has the Haagerup property and $(A, \tau)$ has the Haagerup property, cf.~\cite[Remark 3.10]{mstt}. However, if $G$ is  amenable, then $(C_r^*(\Sigma), \tau\circ E)$ has the Haagerup property if and only if $(A, \tau)$ has the Haagerup property (cf.~\cite[Theorem 2.5]{dong} and \cite[Theorem 3.19]{mw21}). 

\section{Haagerup property for Fell bundles}\label{sec:haag_bundle}

Let $\B=(B_g)_{g\in G}$ be a Fell bundle over $G$ and  $T=(T_g)_{g\in G}$ be a positive definite $\B$-bundle map. 
We recall that 
\begin{equation} \label{T-eq}
    T_g(b)^*T_g(b)\, \leq \, \|T_e\| \,T_e(b^*b) \quad \text{(in $B_e$) for all } g \in G, b\in B_g,
    \end{equation}
cf.~\cite[Proposition 3.6]{bc25}. Let $\psi$ be a state on $B_e$ and $g\in G$. Then we can define a (right) semi-inner product on $B_g$ by 
\[ \langle b, c\rangle_{g, \psi} = \psi(b^*c) \quad \text{for all} \, b,c \in B_g.\]
Setting $N_{g, \psi}:=\{ b \in B_g : \psi(b^*b)=0\}$ and $B_{g, \psi}:= B_g/N_{g, \psi}$, we get an inner product on $B_{g, \psi}$ given by
\[\langle b + N_{g, \psi}, c + N_{g, \psi}\rangle_{g, \psi} = \psi(b^*c) \quad \text{for all} \, b,c \in B_g.\]
In the sequel, we will often write $[b]_{g, \psi}$ instead of $b + N_{g, \psi}$ for elements in $B_{g, \psi}$. We also set
\[ \|b\|_{g, \psi}:= \|[b]_{g, \psi}\|_{g,\psi} = \psi(b^*b)^{1/2} \leq \|b\| \]
for all $g\in G$ and $b\in B_g$. Further,
we let $H_{g, \psi}$ denote the Hilbert space obtained from completing $B_{g, \psi}$, and 
write $\|\cdot\|_{g, \psi}$ for the associated norm. Then the following analog of \cite[Lemma 3.2]{mstt} holds.
\begin{lemma} \label{MSTT} 
Assume $\psi$ is a state on $B_e$ such that $\psi\circ T_e \leq M \,\psi$ for some $M>0$ and let $g\in G$. Then the assignment \[b + N_{g, \psi} \mapsto T_g(b) + N_{g, \psi}\] for each $b \in B_g$ extends to a bounded linear map $\widetilde{T_g}$ on $H_{g, \psi}$ such that
$\|\widetilde{T_g}\| \leq (M\, \|T_e\|)^{1/2}$. In particular, if $\psi\circ T_e \leq \,\psi$ and $T_e$ is contractive, then $\widetilde{T_g}$ is contractive.
\end{lemma}
\begin{proof} Let $g\in G$ and $b \in B_g$.
Then, using our standing assumptions, and the inequality (\ref{T-eq}), we get
\[ \psi(T_g(b)^*T_g(b)) \leq \|T_e\| \psi (T_e(b^*b)) \leq M \|T_e\| \psi(b^*b).\]
This implies that the map $b + N_{g, \psi} \mapsto T_g(b) + N_{g, \psi}$ is well-defined and linear on $B_{g, \psi}$. Moreover, we get that
\[ \|T_g(b) + N_{g, \psi}\|_{g, \psi}^2 = \psi(T_g(b)^*T_g(b)) \leq M\|T_e\|\, \|b + N_{g, \psi}\|_{g, \psi}^2\]
for all $g\in G$ and $b\in B_{g}$, and the conclusion clearly follows.
\end{proof}
Using this lemma, we can make the following definition.
\begin{definition}
 Assume  $T=(T_g)_{g\in G}$ is a $\B$-bundle map and $\psi$ is a state on $B_e$ such that $\psi\circ T_e \leq M \,\psi$ for some $M>0$.
We will say that $T$ is \emph{compact w.r.t.~$\psi$} when each $\widetilde{T_g}$ is a compact operator on $H_{g, \psi}$. 
\end{definition}
\begin{definition}\label{def:haag_bundle}
Let $\B=(B_g)_{g\in G}$ be a Fell bundle over $G$ and $\psi$ be a state on $B_e$.
We will say that the pair \emph{$(\B, \psi)$ has the Haagerup property} if there exists a net $\{T^i\}_{i\in I}$ of positive definite $\B$-bundle maps such that 
\begin{itemize}
\item[i)] $T^i_e$ is contractive and $\psi\circ T^i_e \leq \psi$ for every $i\in I$;
\item[ii)] $T^i$ is compact w.r.t.~$\psi$ for every $i\in I$;
\item[iii)] the map $g\mapsto \big\|\widetilde{T^i_g}\big\|$ vanishes at infinity for each $i\in I$;
\item[iv)] $\lim_i \|T^i_g(b) - b\|_{g, \psi}= 0$ for every $g\in G$ and $b\in B_g$.
\end{itemize}
\end{definition}
\begin{example}\label{ex:haag}
    a)  If $G$ is the trivial group, so that $\B = B_e$, then the Haagerup property for $(\B,\psi)$ coincides with the Haagerup property for $(B_e,\psi)$.
    
    b) If $\B = ({\mathbb C} \times \{g\})_{g \in G}$ is the group bundle over $G$, then the Haagerup property of $(\B,\psi)$ corresponds to the Haagerup property for $G$ ($\psi$ being the unique state on $\mathbb C$).  
    
    c) More generally, if $\Sigma=(A, G, \alpha)$ is a unital discrete $C^*$-dynamical system, $\tau$ is an $\alpha$-invariant state  on $A$ and $\B_\Sigma = (A\times \{g\})_{g\in G}$ denotes the Fell bundle over $G$ canonically associated to $\Sigma$, cf.~\cite{exel}, then one may easily check that $(\B_\Sigma, \tau)$ has the Haagerup property if and only if $(\Sigma, \tau)$ has the Haagerup property. This relies on the tight connection between $\Sigma$-positive definite multipliers and positive definite $\B_\Sigma$-bundle maps (cf.~\cite[Example 4.6]{bc25}).  
\end{example}

The equivalence part of Theorem \ref{HP-dyn} can be generalized to Fell bundles:
\begin{theorem}\label{HP-Fell}
Let $\B=(B_g)_{g\in G}$ be a Fell bundle, $\psi$ be a state on $B_e$ and $E_e: C_r^*(\B)\to B_e$ denote the canonical conditional expectation.
Then $(\B, \psi)$ has the Haagerup property if and only if $(C_r^*(\B), \psi \circ E_e)$ has the Haagerup property.  
\end{theorem}

\begin{proof}
We adapt the proof of \cite[Theorem 3.8]{mstt} to the present setting. 

We introduce some notation. Let $\psi'$ be the state on $C_r^*(\B)$ given by $\psi'=\psi\circ E_e$. For $g\in G$, let $V_g: H_{g,\psi}\to H_{\psi'}$ be the natural isometry 
determined by \[V_g([b]_{g, \psi}) = [\lambda_g^\B(b)]_{\psi'} \quad \text{for all } b \in B_g,\] and let $P_g\in B(H_{\psi'})$ denote the orthogonal projection onto the range of $V_g$. 
 One readily sees that $V_g^*([x]_{\psi'}) = [E_g(x)]_{g, \psi}$ for every $g\in G$ and $x\in C_r^*(\B)$.

    Assume first $\{T^i\}_{i\in I}$ is a net of  positive definite $\B$-bundle maps implementing the Haagerup property for $(\B,\psi)$.  Let $i\in I$. By \cite[Theorem 3.14]{bc25}, the associated completely positive map $M_{T^i}: C_r^*(\B)\to C_r^*(\B)$ satisfies $\|M_{T^i}\|=\|T_e^i\|\leq 1$, and we have $$\psi'(M_{T^i}(x))=(\psi\circ E_e)(M_{T^i}(x))=\psi(T^i_e( E_e(x)))\leq \psi(E_e(x))=\psi'(x),$$ 
 for all $x\in C_r^*(\B)^+$. Next, identifying each $H_{g, \psi}$ with its image in $H_{\psi'}$ via $V_g$, we observe that $$H_{\psi'}=\bigoplus_{g\in G} H_{g,\psi} \text{ and }S_{ M_{T^i}}=\bigoplus_{g\in G} \widetilde{T_g^i}.$$ As $g\mapsto \|\widetilde{T_g^i}\|$ vanishes at infinity and each $\widetilde{T_g^i}$ is compact for every $g\in G$,  we obtain that $S_{M_{T^i}}$ is compact on $H_{\psi'}$. The convergence of $S_{ M_{T^i}}$ to $I_{H_{\psi'}}$ follows easily from density arguments and the fact that $\lim_i\|T_g^i(b)-b\|_{g,\psi}=0$ for all $g\in G$ and $b\in B_g$.

    Conversely, assume that $(C_r^*(\B), \psi')$ has the Haagerup  property and let $\{\Phi_i\}_{i\in I}$ be the corresponding net of contractive completely positive maps on $C_r^*(\B)$. For $i\in I$, $g\in G$ and $b\in B_g$, set \[T_g^i(b):= E_g(\Phi_i(\lambda^\B_g(b))),\] where $E_g(x)$ denotes the $g$-Fourier coefficient of $x \in C_r^*(\B)$. Let $i\in I$. By \cite[Proposition 3.8]{bc25}, $T^i=(T_g^i)_{g\in G}$ is a positive definite $\B$-bundle map. Moreover, as $E_e$ and $\Phi_i$ are contractive,
    \[\|T_e^i(b)\|=\|E_e(\Phi_i(\lambda^\B_e(b)))\|\leq \|\lambda^\B_e(b)\| = \|b\|\] 
    for all $b\in B_e$, i.e., $T^i_e$ is contractive. Since
    $$(\psi\circ T_e^i)(b)=(\psi\circ E_e)(\Phi_i(\lambda^\B_e(b)))\leq (\psi\circ E_e)(\lambda^\B_e(b))=\psi(b)$$
    for all $b\in B_e^+$, we also get $\psi\circ T_e^i\leq\psi$. Thus we can form $\widetilde{T_g^i}\in B(H_{g,\psi})$  for every $g\in G$, cf.~Lemma \ref{MSTT}.
   
 We note that $\widetilde{T_g^i}=V_g^*S_{\Phi_i}V_g$ for every $g\in G$.  Indeed, for all $b\in B_g$, we have
\begin{align*}
(V_g^*S_{\Phi_i}V_g)([b]_{g, \psi})& 
= (V_g^*S_{\Phi_i})([\lambda_g^\B(b)]_{\psi'}) = V_g^*([\Phi_i(\lambda_g^\B(b))]_{\psi'})\\
&=[E_g(\Phi_i(\lambda_g^\B(b)))]_{g, \psi} = [T^i_g(b)]_{g, \psi}
=\widetilde{T_g^i}([b]_{g, \psi}).
\end{align*}
Since $S_{\Phi_i}$ is compact, we see that $\widetilde{T_g^i}=V_g^*S_{\Phi_i}V_g$ is compact.

Consider now a countably infinite subset $\{g_n\}_{n\in\mathbb N}$ of $G$. Then $P_{g_n}$ converges to $0$ strongly as $n\to\infty$. Moreover, as 
 $V_g=P_gV_g$ for every $g\in G$ and $S_{\Phi_i}$ is compact for every $i\in I$, we get that
 $\|\widetilde{T_{g_n}^i}\| \leq \|P_{g_n}S_{\Phi_i}P_{g_n}\|\to  0 $ as ${n\to\infty}$ for every $i\in I$. This implies that the function $g\mapsto \|\widetilde{T_g^i}\|$ vanishes at infinity for all $i\in I$.

 Finally, for every $g\in G$ and $b\in B_g$, 
 \begin{align*}
 \|\widetilde{T_g^i}([b]_{g,\psi})-[b]_{g,\psi}\|_{g,\psi}&=\|V_g^*(S_{\Phi_i}-I)V_g([b]_{g,\psi})\|_{g,\psi}\\
 &\leq \|S_{\Phi_i}([\lambda_g^\B(b)]_{\psi'})-[\lambda_g^\B(b)]_{\psi'}\|_{\psi'}\to_{i} 0,
 \end{align*}
 completing the proof.  
 \end{proof}

\begin{example}\label{CAR-2}
Let $G=\bigoplus_{k=1}^\infty \mathbb{Z}_2$ act on the Cantor set $\Gamma= \prod_{k=1}^\infty
\mathbb{Z}_2$ by translation, and let $\alpha$ denote the associated action of $G$ on $A:=C(\Gamma)$. Set $\Sigma=(A, G, \alpha)$. Then it is well-known that $C_r^*(\Sigma)\simeq M_{2^\infty}$ (see for example \cite[II.10.4.12.(iii)]{bla}). Let $\lambda \in (0,1/2)$ and let $\psi_\lambda$ be the restriction of the Powers state $\omega_\lambda$ on $M_{2^\infty}$ to $A$, so $\omega_\lambda= \psi_\lambda\circ E$, where $E$ is the conditional expectation from $C_r^*(\Sigma)$ onto $A$. Then $\psi_\lambda$ is a faithful tracial state on $A$ (which is not $\alpha$-invariant since $\omega_\lambda$ is not tracial). Moreover, $(A, \psi_\lambda)$ has the Haagerup property (this follows from \cite[3.6]{suzuki13}, but can be proved directly without much trouble), and $(C_r^*(\Sigma), \psi_\lambda\circ E) \simeq (M_{2^\infty}, \omega_\lambda)$ has the Haagerup property, cf.~Example \ref{CAR1}. Hence $(\B_\Sigma, \psi_\lambda)$ has the Haagerup property by Theorem \ref{HP-Fell}.  
\end{example}
 
 One may wonder whether the second part of Theorem \ref{HP-dyn} also holds in the setting of Fell bundles: 
 If $(\B,\psi)$ has the Haagerup property, is it true that
$G$ has the Haagerup property and $(B_e,\psi)$ has the Haagerup property?

While it is not clear that $G$ must have the Haagerup property in general, we can give a direct proof 
that $(B_e,\psi)$ has the Haagerup property. This fact can alternatively be deduced by combining Theorem 
\ref{HP-Fell} with \cite[Proposition 3.8]{mw21}.

\begin{proposition} \label{Haag-descend}
Suppose that $(\B,\psi)$ has the Haagerup property. Then
$(B_e,\psi)$ has the Haage\-rup property. 
\end{proposition}
\begin{proof}
Let $\{T^i\}_{i\in I}$ be the net of positive definite $\B$-bundle maps as in the definition of the Haagerup property for $(\B,\psi)$, and set $\Phi_i := T^i_e$. Then $\{\Phi_i\}_{i\in I}$ is a net of contractive completely positive maps on $B_e$ as in the definition of the Haagerup property of $(B_e,\psi)$. Here, $T^i_e$ is completely positive by \cite[Proposition 3.6]{bc25} and $S_{\Phi_i} = \widetilde{T^i_e}$.
\end{proof}

\medskip The following result was proved in \cite[Corollary 3.9]{mstt} in the case of a Fell bundle associated to a discrete $C^*$-dynamical system  and an invariant tracial state. However, traciality is not needed in this proof, cf.~\cite[Proposition 3.16]{mw21}, and one can also allow a scalar-valued twist. We present the argument for completeness. 

\begin{proposition}\label{HP-group}
    Let $\B$ be a unital crossed product bundle over $G$,  with unitary section $u$ and associated twisted action $(\alpha, \sigma)$ of $G$ on $B_e$ such that $\sigma$ is scalar-valued. Assume that $\psi$ is an $\alpha$-invariant state on $B_e$ such that $(\B,\psi)$ has the Haagerup property. Then $G$ has the Haagerup property.
\end{proposition}
\begin{proof}
Let $\{T^i\}_{i\in I}$ be a net of positive definite $\B$-bundle maps implementing the Haagerup property of $(\B,\psi)$. Then the net $\{\varphi_i\}_{i\in I}$ given by 
     \[\varphi_i(g)=\psi(T_g^i(u(g))u(g)^*)\] 
for $i\in I$ and $g\in G$ will implement the Haagerup property of $G$. 

Indeed, let $i\in I$ and $g_1, \ldots, g_n \in G$. The positive definiteness of $T^i$ gives that the matrix 
 \[\Big[u(g_k)T_{g_k^{-1}g_l}^i(u(g_k)^*u(g_l))u(g_l)^*\Big]_{k,l=1}^n\] 
is positive in $M_n(B_e)$. Since $\psi$ is completely positive, being a state, this implies that  \[\Big[\psi\Big(u(g_k)T_{g_k^{-1}g_l}^i(u(g_k)^*u(g_l))u(g_l)^*\Big)\Big]_{k,l=1}^n\]  is  positive in $M_n(\mathbb{C})$. 

For each $g\in G$, let $\psi_g$ be the state on $B_e$ defined by $\psi_g(a)=\psi(u(g) a u(g)^*)$. Since $\psi$ is $\alpha$-invariant, $\psi=\psi_g$ for every $g\in G$. Thus  \[\psi(u(g) b u(h)^*) = \psi(b u(h)^*u(g))\] for every $g,h\in G$ and $b\in B_{g^{-1}h}$. Then, since $\sigma$ is scalar-valued, we get that
\begin{align*} 
\psi\Big(u(g_k)T_{g_k^{-1}g_l}^i&(u(g_k)^* u(g_l))u(g_l)^*\Big)
= \psi\Big(T_{g_k^{-1}g_l}^i(u(g_k)^*u(g_l))u(g_l)^*u(g_k)\Big)\\
&=\psi\Big(T_{g_k^{-1}g_l}^i\big(\overline{\sigma(g_k, g_k^{-1}g_l)}u(g_k^{-1}g_l)\big)\sigma(g_k, g_k^{-1}g_l)u(g_k^{-1}g_l)^*\Big)\\
&=\psi\Big(T_{g_k^{-1}g_l}^i(u({g_k}^{-1}g_l))u(g_k^{-1}g_l)^*\Big) \\
&= \varphi_i(g_k^{-1}g_l)
\end{align*}
for all $i\in I$ and all $k, l \in \{1, \ldots, n\}$. It follows that $[\varphi_i(g_k^{-1}g_l)]_{k,l}$ is positive in $M_n(\mathbb{C})$,  hence that $\varphi_i$ is positive definite for every $i\in I$. As $T_g(b)^*=T_{g^{-1}}(b^*)$ and $u(g)^*=u(g^{-1})\overline{\sigma(g,g^{-1})}$ for every $g\in G$ and $b\in B_g$, we get 
    \begin{align*}
    &|\varphi_i(g)|=|\psi(T_g^i(u(g))u(g)^*)|=\Big|\langle [T_g^i(u(g))^*]_{g^{-1},\psi},[u(g)^*]_{g^{-1},\psi}\rangle_{g^{-1},\psi}\Big|\\
    &=\Big|\langle [T_{g^{-1}}^i(u(g^{-1}))]_{g^{-1},\psi},[u(g^{-1})]_{g^{-1},\psi}\rangle_{g^{-1},\psi}\Big|\leq\|\widetilde{T_{g^{-1}}^i}\|
    \end{align*}
    for every $g\in G$. Thus, $\varphi_i$ vanishes at infinity for each $i\in I$. Finally, using again the equalities above,  for each $g\in G$, we obtain 
    \begin{align*}
        \big|\varphi_i(g)-1\big|&=\Big|\psi\Big((T_g^i(u(g))-u(g))u(g)^*\Big)\Big|\\
        & = \Big|\Big\langle [(T_g^i(u(g))-u(g))^*]_{g^{-1},\psi}, [u(g)^*]_{g^{-1},\psi}\Big\rangle_{g^{-1}, \psi}\Big| \\ 
        &  \leq \|\widetilde{T_{g^{-1}}^i}([u(g^{-1})]_{g^{-1},\psi})-[u(g^{-1})]_{g^{-1},\psi}\|_{g^{-1},\psi} \ \longrightarrow_{i} 0.
    \end{align*}
\end{proof}
\begin{remark}\label{weak-inv} In the proof above, we only need to assume that $\psi$ is a state on $B_e$
satisfying $\psi(u(h)T^i_g(u(g))u(g)^*u(h)^*)=\psi(T^i_g(u(g))u(g)^*)$ for all $i\in I$ and $g,h\in G$, which is weaker than assuming $\alpha$-invariance of $\psi$.    
\end{remark}

Let $\B$ be a unital Fell bundle over $G$ and $\psi$ be a state on $B_e$. In view of Proposition \ref{Haag-descend}, an interesting problem is to find additional condition(s)  
ensuring that the Haagerup property for $(B_e, \psi)$ implies the Haagerup property for $(\B, \psi)$ (equivalently, for $(C_r^*(\B), \psi\circ E_e)$). 
A natural additional condition is $C^*$-amenability of $\B$. In the special case where $G$ is amenable, the question then becomes: if $(B_e,\psi)$ has the Haagerup property and $G$ is amenable, is it true that $(\B, \psi)$ has the Haagerup property (as it is the case in Example \ref{CAR-2})?

In the general setting this problem looks quite challenging, but we will provide a quite satisfactory answer to it for crossed product bundles in Theorem \ref{crprod-bundle}.

\section{Crossed product bundles and Haagerup properties}\label{cross_product_bundle}

We will  need the following observation (see the proof of \cite[Corollary 3.9]{mstt} for the case where
$\sigma$ is trivial). 
\begin{lemma}\label{posdef}
Let $\cl B$ be a unital crossed product bundle, with unitary section $u$ and associated twisted action $(\alpha, \sigma)$ of $G$ on $B_e$.
    Assume $\psi$ is an $\alpha$-invariant state on $B_e$. Let $\xi\in C_c(G, B_e)$, $T=(T_g)_{g\in G}$ be a positive definite $\B$-bundle map, and define $\varphi: G \to \mathbb{C}$ by
    \[ \varphi(g) = \psi \Big(\sum_{p\in G}\xi(p)^*T_g(\sigma(g,g^{-1}p)^*u_g)\xi(g^{-1}p)u_g^*\sigma(g,g^{-1}p)\Big)
\] for all $g\in G$. Then $\varphi \in C_c(G)$ is positive definite.
    \end{lemma}
    \begin{proof} Set $S={\rm supp}(\xi)$. Then one easily sees that $\varphi(g) = 0$ whenever $g\not\in S\cdot S^{-1}$. Since $S$ is finite, this shows that $\varphi$ is finitely supported.
    
    Next, let $g, h\in G$.  Then computing $\varphi(g^{-1}h)$, using that  $$u_{g^{-1}h}=\alpha_g^{-1}(\sigma(g,g^{-1}h))u_g^*u_h,$$
$$\alpha_g(\sigma(g^{-1}h,h^{-1}p))^*\sigma(g,g^{-1}h)=\sigma(g,g^{-1}p)\sigma(h,h^{-1}p)^*,$$ 
\noindent and the invariance of $\psi$ under every $\alpha_k$, hence also under each ${\rm Ad}(\sigma(k, l))$ for $k,l\in G$, we get
\begin{eqnarray*}
    &&
    \varphi(g^{-1}h)\\
    &&= 
    \psi\Big(\sum_{p\in G}\xi(p)^*T_{g^{-1}h}(\sigma(g^{-1}h,h^{-1}gp)^*u_{g^{-1}h})
    \xi(h^{-1}gp)u_{g^{-1}h}^*\sigma(g^{-1}h,h^{-1}gp)\Big)\\&&=
    \psi\Big(\sum_{p\in G}\xi(g^{-1}p)^*T_{g^{-1}h}(\sigma(g^{-1}h,h^{-1}p)^*u_{g^{-1}h})
    \xi(h^{-1}p)u_{g^{-1}h}^*\sigma(g^{-1}h,h^{-1}p)\Big)\\&&=
    \psi\Big(\sum_{p\in G}\xi(g^{-1}p)^*T_{g^{-1}h}(\sigma(g^{-1}h,h^{-1}p)^*\alpha_g^{-1}(\sigma(g,g^{-1}h))u_g^*u_h)\\&&\qquad\times\xi(h^{-1}p)(\alpha_g^{-1}(\sigma(g,g^{-1}h))u_g^*u_h)^*\sigma(g^{-1}h,h^{-1}p)\Big)\\&&=
    \psi\Big(\sum_{p\in G}\xi(g^{-1}p)^*T_{g^{-1}h}(u_g^*\alpha_g(\sigma(g^{-1}h,h^{-1}p)^*)\sigma(g,g^{-1}h)u_h)\\&&\qquad\times\xi(h^{-1}p)u_h^*\sigma(g,g^{-1}h)^*\alpha_g(\sigma(g^{-1}h,h^{-1}p))u_g\Big)\\&&=
    \psi\Big(\sum_{p\in G}\xi(g^{-1}p)^*T_{g^{-1}h}(u_g^*\sigma(g,g^{-1}p)\sigma(h,h^{-1}p)^*u_h)\\&&\qquad\times\xi(h^{-1}p)u_h^*\sigma(h,h^{-1}p)\sigma(g,g^{-1}p)^*u_g\Big)\\&&=
    \psi\Big(\sum_{p\in G}u_g^*\sigma(g,g^{-1}p)M(g,h,p)\sigma(g,g^{-1}p)^*u_g\Big)\\&&=
    \psi\Big(\sum_{p\in G}\big[\alpha_g^{-1}\circ \text{Ad}(\sigma(g,g^{-1}p))\big]\big(M(g,h,p)\big)\Big)=\psi\Big(\sum_{p\in G} M(g,h,p)\Big)
 \end{eqnarray*}
where
\[M(g,h,p):= \sigma(g,g^{-1}p)^*u_g\xi(g^{-1}p)^*\, T_{g^{-1}h}(u_g^*\sigma(g,g^{-1}p)\sigma(h,h^{-1}p)^*u_h)\,\xi(h^{-1}p)u_h^*\sigma(h,h^{-1}p)
\]

\smallskip \noindent Let ${\bf g}=(g_1,\ldots, g_n)\in G^n$, $\eta=(\eta_1,\ldots,\eta_n)\in \mathbb C^n$. For $p\in G$, set $$a_i(p)=\sigma(g_i,g_i^{-1}p)^*u_{g_i}\in B_{g_i}$$ and $$b_i(p)=\bar\eta_i\sigma(g_i,g_i^{-1}p)^*u_{g_i}\xi(g_i^{-1}p)^*\in B_{g_i}.$$ As $T=(T_g)_{g\in G}$ is positive definite, we have that
$$ h(p):=\sum_{i,j=1}^n b_i(p)T_{g_i^{-1}g_j}(a_i(p)^*a_j(p))b_j(p)^*\in B_e^+ \text{ for all }p\in G.$$ 
Thus
$$\sum_{i,j=1}^n \varphi (g_i^{-1}g_j)\bar\eta_i\eta_j=\sum_{p\in G}\psi(h(p))\geq 0.$$ This shows that $\varphi$ is positive definite.
\end{proof}

For the proof of our next result we need to recall that, as in \cite[section 5.1]{abf},  one may associate to a Fell bundle $\B=(B_{t})_{t\in G}$ a $W^*$-bundle $\B''=(B_t'')_{t\in G}$, where $B_t''$ is the $W^*$-completion of $B_t$ in the bidual $C^*(\B)''$, $C^*(\B)$ being the full cross-sectional $C^*$-algebra of $\B$ as defined in \cite{exel}. Then $B_t''$ is isometrically isomorphic to the bidual of $B_t$ for each $t\in G$. Moreover,  $\B''$ is a Fell bundle over $G$ such that $\B$ is a Fell subbundle of $\B''$, having the property that
the maps $B_t''\to B_{t^{-1}}''$, $b\mapsto b^*$, and $B_t''\to B_{st}''$, $b\mapsto ab$, are weak*-continuous. In what follows we write $Z(A)$ for the center of a $C^*$-algebra $A$. 

The following proposition is known in the context of $C^*$-dynamical systems \cite[Remarque 4.2]{Claire}, also in the twisted case when the cocycle is scalar-valued \cite[Corollary 5.18]{bc12}.
\begin{proposition} \label{invAPamen}
    Let $\cl B$ be a unital crossed product bundle, with unitary section $u$ and associated twisted action $(\alpha, \sigma)$ of $G$ on $B_e$.
    Assume that there exists an $\alpha$-invariant state $\psi$ on $B_e$ and $\B$ is $C^*$-amenable.
    Then $G$ is amenable. 
   \end{proposition}
\begin{proof}
   As $\B$ is $C^*$-amenable, we know that $C^*(\B)\simeq C_r^*(\B)$, cf.~\cite{exel, af19}. Moreover, by \cite[Theorem 6.10]{abf}, there exists a bounded net $\{\xi_i\}_{i\in I}$ in $\ell^2(G,Z(B_e''))$ of functions with finite support such that for $t\in G$ and $b\in B_t$, $$\lim_i\sum_{r\in G}\xi_i(r)^*b\xi_i(t^{-1}r)=b$$
      in $B_t''$ with respect to weak*-topology.  Set 
 $\phi:=\psi\circ E_e:C_r^*(\B)\to \mathbb{C}$ and write  $\phi'': C^*(\B)''\to \mathbb{C}$ for the  bidual.
 For each $g\in G$ let $\tilde\alpha_g$ be the $*$-automorphism of $B_e''$ given by $\tilde\alpha_g(x)=u_gxu_g^*$, cf.~\cite{abf}.
By the $\alpha$-invariance of $\psi$ and the weak*-continuity of $\phi''$ and $\tilde\alpha_g$, we get 
 that $\phi''(\tilde\alpha_g(x))=\phi''(x)$ for all $x\in B_e''$. This implies that $\phi''$ is invariant under ${\rm Ad}(\sigma(g,h)^*)$ for all $g, h\in G$.

Now, for each $i\in I$, define  $\varphi_i:G\to \mathbb{C}$  by
    $$\varphi_i(g)=\phi'' \Big(\sum_{p\in G}
\xi_i(p)^*\sigma(g,g^{-1}p)^*u_g\xi_i(g^{-1}p)u_g^*\sigma(g,g^{-1}p)\Big).$$
As in Lemma \ref{posdef} we see that each $\varphi_i$ is finitely supported and positive definite. Furthermore, using  that $\xi_i\in \ell^2(G, Z(B_e''))$ for each $i\in I$ and 
$$\sum_{p\in G}\xi_i(p)^*u_g\xi_i(g^{-1}p)u_g^*\to_i u_gu_g^* = 1_{B_e}$$ in $B_e''$ w.r.t.~the weak*-topology, we obtain
that for every $g\in G$
\begin{align*}
    \varphi_i(g) &=  \sum_{p\in G}
\phi''(\sigma(g,g^{-1}p)^*\xi_i(p)^*u_g\xi_i(g^{-1}p)u_g^*\sigma(g,g^{-1}p))\\
&=\sum_{p\in G} \phi''(\xi_i(p)^*u_g\xi_i(g^{-1}p)u_g^*)\\
&=\phi'' \Big(\sum_{p\in G}\xi_i(p)^*u_g\xi_i(g^{-1}p)u_g^*\Big)\to_i \phi''(1_{B_e}) = 1,
\end{align*}
hence that $G$ is amenable.
\end{proof}

The first statement in the next lemma generalizes \cite[Proposition 2.11]{mstt}.
\begin{lemma}\label{posdef2} 
Let $\cl B$ be a unital crossed product bundle, with unitary section $u$ and associated twisted action $(\alpha, \sigma)$ of $G$ on $B_e$.
    Assume $\Phi:B_e\to B_e$ is a  completely positive map. Let $\xi\in C_c(G, B_e)$, and define a $\B$-bundle map $T=(T_g)_{g\in G}$ by 
    \[T_g(au_g)=\sum_{p\in G}\xi(p)^*[(\alpha_p\circ\Phi\circ\alpha_p^{-1})(a\sigma(g, g^{-1}p))] \sigma(g, g^{-1}p)^* u_g\xi(g^{-1}p)\]
for all $g\in G$ and $a\in B_e$.  Then $T$ is positive definite.

Moreover, assume that $\psi$ is an $\alpha$-invariant state on $B_e$, $\psi\circ \Phi \leq m \psi$ for some $m>0$ and $\Phi$ induces a compact operator on $H_{\psi}$. 

Then  $\psi\circ T_e \leq M \, \psi$ for some $M>0$ (depending on $\xi$) and $T$ is compact w.r.t.~$\psi$ whenever one of the following conditions holds:
\begin{itemize}
    \item[a)] $\psi$ is tracial;
    \item[b)] $\xi(h)\in Z(B_e)$ for all $h\in G$;
    \item [c)] $\xi(h)$ lies in the multiplicative domain of $\psi$ for all $h\in G$. 
\end{itemize}
    \end{lemma}
    \begin{proof}
 Let $g, h\in G$ and $a, b\in B_e$.  We will use below the following identities, valid for all $p\in G$, that are easily derived from the cocycle identities:
 \begin{itemize}
     \item[$(i)$] $u_g^*u_h = \alpha_g^{-1}(\sigma(g, g^{-1}h)^*)u_{g^{-1}h}$,
     \item[$(ii)$] $\sigma(g,g^{-1}h)^*\alpha_g(\sigma(g^{-1}h,h^{-1}p))\sigma(g,g^{-1}p)=\sigma(h,h^{-1}p)$,
     \item[$(iii)$]
$\alpha_g\circ\alpha_{g^{-1}p}= {\rm Ad}(\sigma(g,g^{-1}p))\circ\alpha_p$.
     \end{itemize}
     We compute
\begin{align*}
  &T_{g^{-1}h}((au_g)^*bu_h)=T_{g^{-1}h}(\alpha_g^{-1}(a^*b)u_g^*u_h) 
  \overset{(i)}{=} 
  T_{g^{-1}h}(\alpha_g^{-1}(a^*b\sigma(g, g^{-1}h)^*) u_{g^{-1}h})\\  
 &= 
\sum_{p\in G}\xi(p)^*(\alpha_p\circ\Phi\circ\alpha_p^{-1})\left(
\alpha_g^{-1}\big(a^*b\sigma(g, g^{-1}h)^*\big)
\sigma(g^{-1}h,(g^{-1}h)^{-1}p)\right)\\&\qquad\times
\sigma(g^{-1}h, (g^{-1}h)^{-1}p)^*u_{g^{-1}h}\xi(h^{-1}gp)\\
&=
\sum_{p\in G}\xi(g^{-1}p)^*(\alpha_{g^{-1}p}\circ\Phi\circ\alpha_{g^{-1}p}^{-1})\left(\alpha_g^{-1}(a^*b\sigma(g, g^{-1}h)^*)\sigma(g^{-1}h,h^{-1}p)\right)\\&\qquad\times\sigma(g^{-1}h, h^{-1}p)^*u_{g^{-1}h}\xi(h^{-1}p)\\&=
\sum_{p\in G}\xi(g^{-1}p)^*\alpha_{g^{-1}p}\left((\Phi\circ\alpha_{p}^{-1})\left(\left[{\rm Ad}(\sigma(g,g^{-1}p)^*)\right]\Big(a^*b\sigma(g, g^{-1}h)^*\alpha_g(\sigma(g^{-1}h,h^{-1}p))\Big)\right)\right)\\&\qquad\times\sigma(g^{-1}h, h^{-1}p)^*\alpha_g^{-1}(\sigma(g,g^{-1}h))u_g^*u_h\xi(h^{-1}p)\\&=
\sum_{p\in G}\xi(g^{-1}p)^*u_g^*\alpha_g\Big(\alpha_{g^{-1}p}\left((\Phi\circ\alpha_{p}^{-1})\left(\left[{\rm Ad}(\sigma(g,g^{-1}p)^*)\right]\Big(a^*b\sigma(g, g^{-1}h)^*\alpha_g(\sigma(g^{-1}h,h^{-1}p))\Big)\right)\right)\\&\qquad\times\sigma(g^{-1}h, h^{-1}p)^*\alpha_g^{-1}(\sigma(g,g^{-1}h))\Big)u_h\xi(h^{-1}p)\\&
\overset{(ii), (iii)}{=}
\sum_{p\in G}\xi(g^{-1}p)^*u_g^*\sigma(g,g^{-1}p)(\alpha_{p}\circ\Phi\circ\alpha_{p}^{-1})(\sigma(g,g^{-1}p)^*a^*b\sigma(h, h^{-1}p))
\sigma(h,h^{-1}p)^*u_h\xi(h^{-1}p).
\end{align*}

As $\alpha_p\circ\Phi\circ\alpha_p^{-1}$ is completely positive for each $p \in G$, 
$$\sum_{i,j=1}^nd_i(\alpha_p\circ\Phi\circ\alpha_p^{-1})(a_i^*a_j)d_j^*\geq 0$$
for all ${\bf a}=(a_1,\ldots, a_n)$ and ${\bf d}=(d_1,\ldots, d_n)\in B_e^n$. 

Let now ${\bf g}=(g_1,\ldots g_n)\in G^n$ and $b_i, c_i\in B_{g_i}$ for $i=1, \ldots, n$. Then, observing that 
$$d_i(p):=b_i\xi(g_i^{-1}p)^*u_{g_i}^*\sigma(g_i,g_i^{-1}p)\in B_e \text{ and }a_i(p):=c_iu_{g_i}^*\sigma(g_i, g_i^{-1}p)\in B_e$$ for all $i=1, \ldots, n$, we obtain 
\begin{align*}
    &\sum_{i,j=1}^n b_iT_{g_i^{-1}g_j}(c_i^*c_j)b_j^* = \sum_{i,j=1}^n b_iT_{g_i^{-1}g_j}\big(((c_iu_{g_i}^*)u_{g_i})^*(c_ju_{g_j}^*)u_{g_j}\big)b_j^*
      \\
     &=\sum_{p\in G}\sum_{i,j=1}^n d_i(p)(\alpha_{p}\circ\Phi\circ\alpha_{p}^{-1})(a_i(p)^*a_j(p))d_j(p)^*\geq 0.
\end{align*}
This shows that $T$ is positive definite. 

\smallskip
Next, assume $\psi$ is an $\alpha$-invariant state on $B_e$ satisfying a), b) or c), and  $\psi\circ \Phi \leq m \psi$ for some $m>0$. 

To show that $\psi\circ T_e \leq M \, \psi$ for some $M>0$, let $a\in B_e^+$. 
Assume a) holds. For each $p\in G$, set $c(p, a):=(\alpha_p\circ\Phi\circ\alpha_{p}^{-1})(a)^{1/2} \in B_e^+$. Then
\begin{align*}
  \psi(T_e(a))&= 
 \sum_{p\in G}  \psi\big(\xi^*(p)c(p,a)^2 \xi(p)\big)
  =\sum_{p\in G} \psi\big(c(p,a)\xi(p)\xi(p)^* c(p,a)\big)\\
  &\leq \sum_{p\in G} \|\xi(p)\|^2 \psi(c(p,a)^2) = \sum_{p\in G} \|\xi(p)\|^2 \psi((\alpha_p\circ\Phi\circ\alpha_{p}^{-1})(a))\\
  &\leq m \Big(\sum_{p\in G} \|\xi(p)\|^2\Big)\, \psi(a),
\end{align*}
so we can choose $M=m\big(\sum_{p\in G} \|\xi(p)\|^2\big)$.
Cases b) and c) can be handled in a similar way.

\smallskip
Assume further that $\Phi$ induces a compact operator on $H_{\psi}$.
Let $g\in G$. To show that $\widetilde{T_g}$ is compact on $H_{g, \psi}$,
we first observe that Lemma \ref{MSTT} gives that it is well-defined and bounded.
Fix $p\in G$. By additivity, it suffices to show that the bounded 
operator $K$ on $H_{g,\psi}$ determined by
\[[b]_{g, \psi}\mapsto \Big[\xi(p)^*(\alpha_p\circ\Phi\circ\alpha_p^{-1})(bu_g^*\sigma(g,g^{-1}p))\sigma(g,g^{-1}p)^*u_g\xi(g^{-1}p)\Big]_{g, \psi}\] 
is compact.
 
 We start by noticing that for each $g\in G$ the map $[b]_{g, \psi} \mapsto [bu_g^*\sigma(g,g^{-1}p)]_{\psi}$, $b\in B_g$, 
 is well-defined and induces an isometric linear map $R_1$ from $H_{g, \psi}$ to $H_{\psi}$. This follows readily from the $\alpha$-invariance of $\psi$, as it implies that
\[
 \psi\big((bu_g^*\sigma(g,g^{-1}p))^*bu_g^*\sigma(g,g^{-1}p)\big) =
 \psi\big({\rm Ad}(\sigma(g,g^{-1}p)^*) \alpha_g(b^*b)\big) = \psi(b^*b)
\]
for all $b\in B_g$. Similarly, the map $[a]_\psi \mapsto [a\sigma(g,g^{-1}p)^*u_g]_{g,\psi}$, $a\in B_e$, 
is well-defined and induces an isometric linear map $R_2$ from  $H_{\psi}$ to $H_{g, \psi}$.  
The $\alpha$-invariance of  $\psi$ also implies that $\alpha_p$ induces a unitary map $R_3$
on $H_\psi$.

Next, we  observe that the map  $[b]_{g, \psi}\mapsto [\xi(p)^*b\xi(g^{-1}p)]_{g, \psi}$, $b\in B_g$,  is well-defined 
 and extends to a bounded linear map $R_4$ on $H_{g, \psi}$ with $\|R_4\| \leq \|\xi(p)\|\,\|\xi(g^{-1}p)\|$.
 To see this, let $b\in B_g$. Then we have  
 \begin{align*}
 \psi\big((\xi(p)^*b\xi(g^{-1}p))^*\xi(p)^*b\xi(g^{-1}p)\big)& =    
 \psi\big(\xi(g^{-1}p)^*b^*\xi(p)\xi(p)^*b\xi(g^{-1}p)\big)\\
&\leq \|\xi(p)\|^2 \|\xi(g^{-1}p) \|^2 \psi(b^*b)
\end{align*}
 the last inequality above being easily verified by using condition a), b) or c), from which the claim follows readily. 

Now, the assumption says that $\Phi$ extends to a compact operator $S_\Phi$ on $H_\psi$.   
As $K$ can be written as
\[K = R_4\circ R_2\circ R_3\circ S_\Phi\circ R_3^{-1}\circ R_1,\]
we get that $K$ is compact on $H_{g, \psi}$ (with $\|K\| \leq \|\xi(p)\|\,\|\xi(g^{-1}p)\|\, \|S_\Phi\|$), as desired.
 \end{proof}

    \begin{theorem} \label{crprod-bundle}
  Let $\cl B$ be a unital crossed product bundle over $G$ with unitary section $u$ and associated twisted action $(\alpha, \sigma)$ of $G$ on $B_e$. 
  
 Assume that $\B$ is $C^*$-amenable and that there exists an  $\alpha$-invariant state $\psi$ on $B_e$ such that $(B_e,\psi)$ has the Haagerup property.

Then $G$ is amenable, $(\mathcal B, \psi)$ has the Haagerup property and $(C_r^*(\B), \psi\circ E_e)$ has the Haagerup property.
\end{theorem}
\begin{proof}
The proof is an adaptation of arguments in \cite[Corollary 4.6]{mstt} (see also \cite[Corollary 3.4]{mt}).
Let $\{\Phi_i\}_{i\in I}$ be an approximating net of completely positive maps on $B_e$ that implements the Haagerup property of $(B_e,\psi)$.
Thanks to Proposition \ref{invAPamen}, 
$G$ is amenable. Hence, in Definition \ref{ap},  the net $\{\xi_j\}_{j\in J}$ can be chosen to lie in $C_c(G,\mathbb{C} 1_{B_e})$ (cf.~the proof of \cite[Theorem 20.4]{exel}). Letting $\{F_j\}_{j\in J}$ be a F\o lner net of non-empty finite subsets of $G$, we can even choose 
\[\xi_j(g):=\frac{1}{|F_j|^{1/2}}\chi_{F_j}(g) 1_{B_e}\]
for all $j\in J$ and $g\in G$. Set
\begin{align*}
    T_g^{i,j}(au_g)&=\sum_{p\in G}\xi_j(p)^*[(\alpha_p\circ\Phi_i\circ\alpha_p^{-1})(a\sigma(g, g^{-1}p))] \sigma(g, g^{-1}p)^* u_g\xi_j(g^{-1}p)\\
    &= \frac{1}{|F_j|}\sum_{p\in F_j\cap gF_j} [(\alpha_p\circ\Phi_i\circ\alpha_p^{-1})(a\sigma(g, g^{-1}p))] \sigma(g, g^{-1}p)^* u_g
\end{align*}
for all $a\in B_e$ and $ g\in G$. By Lemma \ref{posdef2},  $(T_g^{i,j})_{g\in G}$ is a positive definite $\cl B$-bundle map for all $i\in I$, $j\in J$.

Let $i\in I, j\in J$.
 As $\sigma(e,p)=1$, for each $p\in G$, we have 
$$T_e^{i,j}(a)=\frac{1}{|F_j|}\sum_{p\in F_j}(\alpha_p\circ\Phi_i\circ\alpha_{p}^{-1})(a), \quad a\in B_e.$$
As $\Phi_i$ is contractive, $T_e^{i,j}$ is contractive. Furthermore, for $a\in B_e^+$, using that $\psi$ is $\alpha$-invariant and $\psi\circ\Phi_i\leq \psi$, we get
\begin{align*}
(\psi\circ T_e^{i,j})(a)&=\psi\Big(\frac{1}{|F_j|}\sum_{p\in F_j} \alpha_p(\Phi_i(\alpha_{p}^{-1}(a)))\Big)=\frac{1}{|F_j|}\sum_{p\in F_j}\psi(\alpha_p(\Phi_i(\alpha_{p}^{-1}(a)))) \\&\leq 
\frac{1}{|F_j|}\sum_{p\in F_j}\psi(a) = \psi(a),
\end{align*}
showing that $\psi\circ T_e^{i,j} \leq \psi$. 

 Now, Lemma \ref{posdef2} gives that $\widetilde T_g^{i,j}$ is  compact  on $H_{g,\psi}$. Moreover, it follows from the proof of Lemma \ref{posdef2} that for each $g \in G$,
$$\|\widetilde{T}_g^{i,j}\| \leq \sum_{p\in G} \|\xi_j(p)\|\, \|\xi_j(g^{-1}p)\| 
= \frac{|F_j\cap g F_j|}{|F_j|},$$ 
which implies that $g \mapsto \|\widetilde{T}_g^{i,j}\| \in C_c(G)$.

Next, we observe that for each $g\in G$ and $b\in B_g$
\begin{align*}
    \|T_g^{i,j}(b)-b\|_{ g,\psi}
    &=\Big\|\frac{1}{|F_j|}\sum_{p\in F_j\cap gF_j}(\alpha_p\circ\Phi_i\circ\alpha_{p}^{-1})(bu_g^*\sigma(g,g^{-1}p))\sigma(g,g^{-1}p)^*u_g-b\Big\|_{g,\psi}\\
    &=\Big\|\frac{1}{|F_j|}\sum_{p\in F_j\cap gF_j}(\alpha_p\circ\Phi_i\circ\alpha_{p}^{-1})(bu_g^*\sigma(g,g^{-1}p))\sigma(g,g^{-1}p)^*-bu_g^*\Big\|_{\psi}\\
    &\leq \Big\|\frac{1}{|F_j|}\sum_{p\in F_j\cap gF_j}(\alpha_p\circ\Phi_i\circ\alpha_{p}^{-1})(bu_g^*\sigma(g,g^{-1}p))\sigma(g,g^{-1}p)^*-\frac{|F_j\cap gF_j|}{|F_j|}bu_g^*\Big\|_{\psi}\\
    &\ \ \ \  +\Big\|\frac{|F_j\cap gF_j|}{|F_j|}bu_g^*-bu_g^*\Big\|_{\psi}\\
    &=\Big\|\frac{1}{|F_j|}\sum_{p\in F_j\cap gF_j}\Big((\alpha_p\circ\Phi_i\circ\alpha_{p}^{-1})(bu_g^*\sigma(g,g^{-1}p))-bu_g^*\sigma(g,g^{-1}p)\Big)\sigma(g,g^{-1}p)^*\Big\|_{\psi}\\
    &\ \ \ \  + \Big|\frac{|F_j\cap gF_j|}{|F_j|} -1\Big| \, \|bu_g^*\|_{\psi}\\
    &\leq \frac{1}{|F_j|}\sum_{p\in F_j\cap gF_j}\|(\alpha_p\circ\Phi_i\circ\alpha_{p^{-1}})(bu_g^*\sigma(g,g^{-1}p))-bu_g^*\sigma(g,g^{-1}p)\|_\psi\\
    &\ \ \ \  + \Big|\frac{|F_j\cap gF_j|}{|F_j|} -1\Big| \, \|bu_g^*\|_{\psi}\\
    &
    \leq \frac{1}{|F_j|}\sum_{p\in F_j\cap gF_j} \|(\Phi_i(\alpha_{p}^{-1}(bu_g^*\sigma(g,g^{-1}p)))-\alpha_{p}^{-1}(bu_g^*\sigma(g,g^{-1}p))\|_\psi
    \\
   &\ \ \ \  + \Big|\frac{|F_j\cap gF_j|}{|F_j|} -1\Big| \, \|bu_g^*\|_{\psi}.
\end{align*}

Consider now a finite subset $\mathcal{F}$ of $B_e$, a finite subset $\mathcal{G}$ of $G$ and $\varepsilon>0$. Since $\lim_j\frac{|F_j\cap gF_j|}{|F_j|} = 1$, we can pick $j_0 = j_0(\mathcal{F}, \varepsilon)\in J$ such that 
\[ \Big|\frac{|F_{j_0}\cap gF_{j_0}|}{|F_{j_0}|} -1\Big| \, \|a\|_\psi < \frac{\varepsilon}{2}\]
for all $a \in \mathcal{F}$.

Further, we can pick $i_0 =i_0(\mathcal{F}, \mathcal{G}, \varepsilon)\in I$ such that 
\[ \|(\Phi_{i_0}(\alpha_{p^{-1}}(a\sigma(g,g^{-1}p)))-\alpha_{p^{-1}}(a\sigma(g,g^{-1}p))\|_\psi < \frac{\varepsilon}{2}\]
for all $p \in F_{j_0}$, $g\in \mathcal{G}$ and  $a \in \mathcal{F}$. Then our observation above gives that
\begin{align*}
    \|T_g^{i_0,j_0}(au_g)-au_g\|_{g,\psi} \, &\leq \frac{1}{|F_{j_0}|}\sum_{p\in F_{j_0}\cap gF_{j_0}} \|(\Phi_i(\alpha_{p}^{-1}(a\sigma(g,g^{-1}p)))-\alpha_{p}^{-1}(a\sigma(g,g^{-1}p))\|_\psi\\
    & \ \ \ \  + \Big|\frac{|F_{j_0}\cap gF_{j_0}|}{|F_{j_0}|} -1\Big| \, \|a\|_{\psi}\\
    & \, < \frac{1}{|F_{j_0}|}|F_{j_0}\cap gF_{j_0}|\cdot \varepsilon/2 + \varepsilon/2 \, \leq \, \varepsilon
 \end{align*} 
for all $a\in \F_e$ and $g\in \mathcal{G}$. Set $T^{(\mathcal{F}, \mathcal{G}, \varepsilon)} := (T^{i_0,j_0}_g)_{g\in G}$.

If $\mathcal{F}'$ is a finite subset of $B_e$, $\mathcal{G}'$ is a finite subset of $G$ and $\varepsilon' >0$, set $(\mathcal{F}, \mathcal{G}, \varepsilon)\leq (\mathcal{F}', \mathcal{G}', \varepsilon')$ whenever $\varepsilon\geq \varepsilon'$, $\mathcal{F}\subseteq \mathcal{F}'$ and $\mathcal{G}\subseteq \mathcal{G}'$.
We can then consider $\big(T^{(\mathcal{F}, \mathcal{G}, \varepsilon)}\big)_{(\mathcal{F}, \mathcal{G}, \varepsilon)}$ as a net of positive definite $\B$-bundle maps. 

Let $g\in G$ and $b \in B_g$.  Then we have
\[\|T_g^{(\mathcal F, \mathcal{G}, \varepsilon)}(b)-b\|_{g,\psi}\ \underset{(\mathcal F, \mathcal{G}, \varepsilon)}{\longrightarrow} 0.\] 
Indeed, let $\varepsilon >0$ and write $b=au_g$ with $a\in B_e$. Set $\mathcal{F} := \{a\}$, $\mathcal{G}:=\{g\}$ and assume $(\mathcal{F}', \mathcal{G}', \varepsilon')\geq (\mathcal{F}, \mathcal{G}, \varepsilon)$. Then
\[\|T_g^{(\mathcal{F}', \mathcal{G}', \varepsilon')}(au_g)-au_g\|_{g,\psi}< \varepsilon'\leq\varepsilon.\]

It is now clear that the net $\big(T^{(\mathcal{F}, \mathcal{G}, \varepsilon)}\big)_{(\mathcal{F}, \mathcal{G}, \varepsilon)}$ satisfies all the properties needed to ensure that $(\B, \psi)$ has the Haagerup property. It follows from Theorem \ref{HP-Fell} that $(C_r^*(\B), \psi\circ E_e)$ has the Haagerup property.
\end{proof}
\begin{corollary}
  Let $\cl B$ be a unital crossed product bundle over $G$ with unitary section $u$ and associated twisted action $(\alpha, \sigma)$ of $G$ on $B_e$. 
  
 Assume that $G$ is amenable and let $\psi$ be an  $\alpha$-invariant state on $B_e$. Then 
 $(B_e,\psi)$ has the Haagerup property if and only if $(\B, \psi)$ has the Haagerup property if and only if $(C_r^*(\B), \psi\circ E_e)$ has the Haagerup property.  
\end{corollary}
\begin{proof}
    Follows by combining Theorem \ref{HP-Fell}, Proposition \ref{Haag-descend} and Theorem \ref{crprod-bundle}.
\end{proof}
\begin{remark}
    If one removes the assumption that the state $\psi$ on $B_e$ is $\alpha$-invariant from the assumptions in Theorem \ref{crprod-bundle}, then one can not conclude that $G$ is amenable
    and, even if $G$ is assumed to be amenable from the start, it is not possible proceed in the same way to reach the conclusion.
    It would be interesting to know whether  it is still true for unital crossed product bundles that $(\mathcal B, \psi)$ has the Haagerup property (or, equivalently, that
    $(C_r^*(\B), \psi\circ E_e)$ has the Haagerup property) if one only assumes that $\B$ is $C^*$-amenable and $(B_e, \psi)$ has the Haagerup property. (See Example \ref{CAR-2}
    where all these properties hold and $\psi$ is not $\alpha$-invariant.) 
     Note that this problem makes sense for a general Fell bundle $\B$, where $\alpha$-invariance has no meaning.
\end{remark} 
\begin{remark}
In \cite{gm22}, Gao and Meng introduce a notion of quasi-amenability for an action $\alpha$ of $G$ on a unital $C^*$-algebra $A$ with respect to some $\alpha$-invariant $\tau \in S(A)$. If $\Sigma=(A, G, \alpha)$ denotes the associated system and $E$ is the canonical conditional expectation from $C_r^*(\Sigma)$ onto $A$, then they show in \cite[Theorem 3.5]{gm22} that if $\alpha$ is a quasi-amenable action of $G$ on $A$ w.r.t.~some $\alpha$-invariant $\tau\in S(A)$, then $(A, \tau)$ has the Haagerup property if and only if $(C_r^*(\Sigma), \tau\circ E)$
has it too. It is clear that quasi-amenability of $\alpha$ is a stronger notion than the amenability of $\alpha$ as defined in \cite[Definition 4.3.1]{bo}, which itself implies that the associated Fell bundle $\B_\Sigma$ has Exel's AP, hence is $C^*$-amenable. We note that it therefore follows from Proposition \ref{invAPamen}
that $G$ has to be amenable whenever $\alpha$ is  quasi-amenable w.r.t.~some invariant state on $A$.
Thus, \cite[Theorem 3.5]{gm22} actually only covers the case where $G$ is assumed to be amenable.
\end{remark}

\section{Haagerup PD-approximation property for Fell bundles}
We introduce a variant of the PD-approximation property for Fell bundles taking inspiration from the difference between amenability and the Haagerup property for groups.

 \begin{definition} \label{HPD} A Fell bundle $\B=(B_g) _{g\in G} $ over a  discrete group $G$ is said to   have the \emph{Haagerup PD-approximation property} if there exists a net $\{T^i\}_{i\in I}$ of $\B$-bundle maps satisfying the following properties:
 \begin{itemize}
 \item[(i)] For each $i \in I$, $T^i$ is positive definite  and $g \mapsto \|T^i_g\| \in C_0(G)$;
 \item[(ii)]  $\{T^i\}_{i\in I}$ is uniformly bounded in the sense that $\sup_{i}\|T^i_e\| < \infty$;
 \item[(iii)] $\lim_i \|T^i_g(b) - b\| = 0$ for every $g \in G$ and $b\in B_g$. 
 \end{itemize} 
\end{definition}
Note that $\B$ has the Haagerup PD-approximation property whenever $G$ has the Haagerup property. 
Indeed, if $\{\phi_i\}_{i\in I}$ is  a net of positive definite functions on $G$ giving the Haagerup property for $G$, then the associated net $\{T^{\phi_i}\}_{i\in I}$ of positive definite $\B$-bundle maps defined in \cite[Example 4.4]{bc25} is readily seen to give the Haagerup PD-approximation property for $\B$. Note also that $\B$ has the Haagerup PD-approximation property whenever $\B$ 
is  C*-amenable (since the PD-approximation property, cf.~Definition \ref{PD}, is clearly a stronger property). 
As we will see below, the converse implication does not necessarily hold (see  the discussion after Example \ref{example}).

Let us now verify that Definition \ref{HPD} is compatible with the definition of the Haagerup property for $\Sigma$.
 
  \begin{proposition} \label{SigmaFell}
 Assume  $\Sigma = (A,G,\alpha)$ is a unital discrete $C^*$-dynamical system having the Haagerup property (in the sense of Dong and Ruan). Then the associated Fell bundle $\B_\Sigma$  has the Haagerup PD-approximation property.
\end{proposition}
\begin{proof} Let $\{\varphi_i\}_{i\in I}$ is  a net of normalized  $\alpha$-positive definite functions in $C_0(G,Z(A))$ such that $\lim_{i} \|\varphi_i(g)-1_A\| = 0$ for every $g\in G$.
For each $i\in I$, let $\tilde T^{i}$ denote the associated $\Sigma$-positive definite multiplier of $\Sigma$ given by $\tilde T^{i}(a, g) = \varphi_i(g)a$ for all $(a,g) \in A\times G$, cf.~\cite[Proposition 4.24]{bc16}. Then $T^i=(T^i_g)_{g\in G}$, defined by \[T^i_g((a,g))=(\tilde T^{i}(a, g), g) = (\varphi_i(g)a, g)\]
for all $a \in A$ and $g\in G$, is a positive definite $\B_\Sigma$-bundle map, cf.\cite[Example 4.6]{bc25}.

Moreover, $\|T^i_g\| = \|\varphi_i(g)\|$. Thus, since $\varphi_i\in C_0(G, Z(A))$, we get that $g\mapsto\|T^i_g\|$ lies in $C_0(G)$.  Since $\varphi_i(e) = 1_A$, we also have 
$\|T^i_e\| = \|{\rm id}_A\| = 1$. Hence, $\sup_{i\in I} \|T^i_e\| = 1$.

Finally, for $a\in A$ and $g\in G$, we have that
\[ \lim_{i}\|T^i_g(a, g) - (a, g)\| = \lim_{i}\|(\varphi_i(g)a -a, g)\| = \lim_{i}\|\varphi_i(g)a -a\| = 0.\]
This shows that $\{T^i\}_{i\in I}$ is a net having the desired properties. 
\end{proof}

\medskip One may wonder if it is possible to exhibit a Fell bundle $\B=(B_g)_{g \in G}$ with the Haagerup PD-approximation property such that $G$ does not have the Haagerup property. 
Inspecting the proof of Proposition \ref{HP-group} we arrive at the following partial answer.
\begin{proposition}\label{HPD-Hgroup}
    Let $\B$ be a unital crossed product bundle over $G$,  with unitary section $u$ and associated twisted action $(\alpha, \sigma)$ of $G$ on $B_e$ such that $\sigma$ is scalar-valued. 
    Assume $\B$ has the Haagerup PD-approximation property and there is an  $\alpha$-invariant state $\psi$ on $B_e$.
    Then $G$ has the Haagerup property.
\end{proposition}
 \begin{proof} Let $\{T^i\}_{i\in I}$ be a net implementing the Haagerup PD-approximation property for $\B$ and set $\varphi_i(g):=\psi(T^i_g(u_g)u_g^*)$ for all $g\in G$ and $i\in I$. Then we can argue in a similar way as in the proof of Proposition \ref{HP-group}.
\end{proof}
\begin{remark} As in Remark \ref{weak-inv}, we only need to assume in the statement of Proposition \ref{HPD-Hgroup} that the state $\psi$ on $B_e$ satisfies \[\psi(u_hT^i_g(u_g)u_g^*u_h^*) =\psi(T^i_g(u_g)u_g^*)\] 
for all $i\in I$ and  $g, h \in G$. Of course, this condition holds (for all the states $\psi$ on $B_e$) whenever $T^i_g(u_g)u_g^*$ lies in $B_e^\alpha$, the fixed-point algebra of $B_e$ under $\alpha$, for every $i\in I$ and $g\in G$.
    \end{remark}
 
\begin{example} \label{example} 
We give here an example of a unital discrete $C^*$-dynamical system $\Sigma=(A, G, \alpha)$ with a faithful tracial $\alpha$-invariant state $\tau$ on $A$ such that 
\begin{itemize}
\item $\B_\Sigma$ has the Haagerup PD-approximation property,
\item $(\B_\Sigma, \tau)$ does not have the Haagerup property,
\item $\B_\Sigma$ is not $C^*$-amenable,
\item $(A, \tau)$ has the Haagerup property.
 \end{itemize}
 
As in \cite[Remark 3.10]{mstt}, we consider the natural action of $SL(2, \mathbb{Z})$ by automorphisms on $\mathbb{Z}^2$  and let $\alpha$ be the induced action of $SL(2, \mathbb{Z})$ on $A:=C_r^*(\mathbb{Z}^2)$. Let $\tau$ be the canonical tracial state on $C_r^*(\mathbb{Z}^2)$. It is well-known (and easy to check) that $\tau$ is $\alpha$-invariant. Since $SL(2, \mathbb{Z})$ has the Haagerup property, the Fell bundle $\B_\Sigma$ associated to the system $\Sigma:=(A, SL(2, \mathbb{Z}), \alpha)$ has the Haagerup PD-approximation property. Moreover, $(\B_\Sigma, \tau)$ does not have the Haagerup property. 
Indeed, as the semi-direct product  $\mathbb{Z}^2 \rtimes SL(2, \mathbb{Z})$ does not have the Haagerup property (cf.~\cite{ccjjv}), the pair $(C_r^*(\mathbb{Z}^2 \rtimes SL(2, \mathbb{Z})), \tau')$ does not have the Haagerup property (cf.~\cite[Theorem 2.6]{dong}), where $\tau'$ denotes the canonical tracial state on $C_r^*(\mathbb{Z}^2 \rtimes SL(2, \mathbb{Z}))$. Letting $E: C_r^*(\Sigma)\to A$ denote the canonical conditional expectation, we can identify $(C_r^*(\mathbb{Z}^2 \rtimes SL(2, \mathbb{Z})), \tau')$ with $(C_r^*(\Sigma), \tau\circ E)$, and obtain that $(C_r^*(\Sigma), \tau\circ E)$ does not have the Haagerup property. 
By \cite[Theorem 3.8]{mstt} this implies that $(\Sigma, \tau)$, hence $(\B_\Sigma, \tau)$, does not have the Haagerup property. 
Further, $\B_\Sigma$ is not $C^*$-amenable. Indeed, assume on the contrary that $\B_\Sigma$ is $C^*$-amenable. 
Then, since $A=C_r^*(\mathbb{Z}^2)$ is nuclear, $C_r^*(\B_\Sigma)\simeq C_r^*(\mathbb{Z}^2\rtimes SL(2,\mathbb{Z}))$ is nuclear too (cf.~\cite[Theorem 5.7]{bc25}, \cite[Theorem 3.20]{bf25}). By Lance's result (see for example \cite[Theorem 2.6.8]{bo}), this implies that $\mathbb{Z}^2\rtimes SL(2, \mathbb{Z})$ is amenable, which is not the case.
Finally, $(A, \tau)$ has the Haagerup property since $\mathbb{Z}^2$ has the Haagerup property.
\end{example}

\medskip
In \cite[Section 2]{dong_ruan} the authors  defined the notion of Hilbert $A$-module Haagerup property w.r.t.~$E$ for a triple $(C, A, E)$, where $A \subset C$ are unital $C^*$-algebras and $ E: C\to A$ is a faithful conditional expectation onto $A$. (See also \cite[Definition 5.12]{kls} for a definition covering the non-unital case, which is consistent with Dong and Ruan's definition 
for unital inclusions when $E$ is tracial, cf.~\cite[Proposition 5.16]{kls}).

Let $\cl B=(B_g)_{g\in G}$ be a unital Fell bundle. Knowing that $E_e:C_r^*(\cl B)\to B_e$ is a faithful conditional expectation onto $E_e$, we can consider the $B_e$-module Haagerup property w.r.t.~$E_e$ for the triple $(C_r^*(\cl B),B_e,E_e)$ and see how it is related to the Haagerup PD-approximation property for $\cl B$.
We recall first some facts necessary to introduce this notion. On $C_r^*(\cl B)$ define a $B_e$-valued inner product by 
\[\langle x,y\rangle_{E_e}=E_e(x^*y), x,y\in C_r^*(\cl B)\] and norm $\|x\|_{E_e}=\|\langle x,x\rangle_{E_e}\|^{1/2} = E_e(x^*x)^{1/2} \leq \|x\|_r$. It induces a pre-Hilbert right $B_e$-module structure on $C_r^*(\cl B)$. Write $\cl H_{\cl B}$ for the right Hilbert $B_e$-module completion.
Let $\Phi: C_r^*(\cl B)\to C_r^*(\cl B)$ be a completely positive map. Then by the  Schwarz inequality,
\[\Phi(x)^*\Phi(x)\leq \|\Phi\|\, \Phi(x^*x), \, x\in C_r^*(\cl B).\]
If in addition $E_e\circ \Phi\leq E_e$, we obtain
\[\|\Phi(x)\|_{E_e}^2=E_e(\Phi(x)^*\Phi(x))\leq \|\Phi\|\, E_e(\Phi(x^*x))\leq \|\Phi\|\,E_e(x^*x)= \|\Phi\|\,\|x\|_{E_e}^2.\]
Hence $\Phi$ determines a bounded map $\tilde\Phi$ on $\cl H_{\cl B}$. 

Fix $\xi$, $\eta\in \cl H_{\cl B}\setminus \{0\}$ and define a rank one map $\theta_{\xi,\eta}: \cl H_{\cl B}\to \cl H_{\cl B}$ by letting
$$\theta_{\xi,\eta}(x)=\xi\langle \eta,x\rangle_{E_e}, x\in \cl H_{\cl B}.$$
Clearly, $\theta_{\xi,\eta}$ is a right $B_e$-module map. A finite rank map on $\cl H_{\cl B}$ is defined to be a finite sum of such maps,
while  a compact map on $\cl H_{\cl B}$ is a bounded map which is a norm limit of finite rank maps. 
We have that any compact map on $\cl H_{\cl B}$ is a right $B_e$-module map.

Inspired by Dong and Ruan's notion of Hilbert module Haagerup property, we propose the following analogous concept. 
\begin{definition}\label{def_module_HP}
We say that $C_r^*(\cl B)$ has the {\it Hilbert $B_e$-module Haagerup property} with respect to $E_e$ if there exists a net of completely positive  $B_e$-bimodule maps $\{\Phi_i\}_{i\in I}$ on $C_r^*(\cl B)$ such that
\begin{enumerate}
    \item[a)] $E_e\circ \Phi_i\leq E_e$ for every $i\in I$;
    \item[b)] $\sup_{i\in I}\|\tilde \Phi_i\|<\infty$;
    \item[c)] $\tilde\Phi_i$ is a compact map on $\cl H_{\cl B}$ for every $i\in I$;
    \item[d)] $\|\tilde\Phi_i(x)-x\|_{E_e}\to_i 0$ for all $x\in C_r^*(\cl B)$.
\end{enumerate}
\end{definition}

\begin{example} \label{DR-Hp} Let $\Sigma=(A, G, \alpha)$ be a unital discrete $C^*$-dynamical system having the Haagerup property (in the sense of Dong and Ruan). Then $C_r^*(\Sigma)$ has the Hilbert $A$-module Haagerup property w.r.t.~$E$, where $E:C_r^*(\Sigma)\to A$ denotes the canonical conditional expectation, cf.~\cite[Theorem 3.6]{dong_ruan}. It is easy to see that this implies that $C_r^*(\B_\Sigma)$ has the Hilbert $B_e$-module Haagerup property w.r.t.~$E_e$.
\end{example}
\begin{theorem} \label{moduleHaag} Let $\cl B=(B_g)_{g\in G}$ be a unital Fell bundle over a discrete group $G$.  Assume that $C_r^*(\cl B)$ has the Hilbert $B_e$-module Haagerup property with respect to $E_e$. Then $\cl B$ has the Haagerup PD-approximation property.
\end{theorem}

\begin{proof}
    Let $\{\Phi_i\}_{i\in I}$  be a net of completely positive maps that implements the Hilbert $B_e$-module Haagerup property w.r.t.~$E_e$. For each $i\in I$ and $g\in G$ set
    \[T_g^i=E_g\circ\Phi_i\circ \lambda_g^\B : B_g\to B_g.\]
    By \cite[Proposition 3.8]{bc25}, $T^i=(T_g^i)_{g\in G}$ is a positive definite $\B$-bundle map for each $i\in I$, and, for each $b\in B_g$, using \cite[Proposition 17.15]{exel} and d), we get 
    \begin{align*}
    &\|T_g^i(b)-b\|=\|E_g\big(\Phi_i(\lambda_g^{\cl B}(b))-\lambda_g^{\cl B}(b)\big)\|\\&=\|E_g\big(\Phi_i(\lambda_g^{\cl B}(b))-\lambda_g^{\cl B}(b)\big)^*E_g\big(\Phi_i(\lambda_g^{\cl B}(b))-\lambda_g^{\cl B}(b)\big)\|^{1/2}\\&\leq\|E_e\big((\Phi_i(\lambda_g^{\cl B}(b))-\lambda_g^{\cl B}(b))^*(\Phi_i(\lambda_g^{\cl B}(b))-\lambda_g^{\cl B}(b))\big)\|^{1/2}\\&= \|\tilde\Phi_i(\lambda_g^{\cl B}(b))-\lambda_g^{\cl B}(b)\|_{E_e}\to_i 0.
    \end{align*} 
    Fix now $i\in I$ and $\varepsilon>0$. As $\tilde \Phi_i$ is compact on $\cl H_{\cl B}$, there is a finite rank map $R$ on $\cl H_{\cl B}$ such that $\|\tilde\Phi_i-R\|<\varepsilon$. 
    
    By density of ${\rm Span}\{\lambda_g^\B(c_g): g\in G \text{ and } c_g \in B_g\}$ in $\mathcal{H}_\B$ we may further assume that there exist $g_1, h_1 \ldots, g_m, h_m$ in $G$ such that $R$ is of the form $\sum_{j=1}^m \theta_{c_j,d_j}$ where $c_{j}\in \lambda_{g_j}^{\cl B}(B_{g_j})$, $d_j\in \lambda_{h_j}^{\cl B}(B_{h_j})$ for all $j=1,\ldots, m$. Set $F=\{h_j \ | \ j=1,\ldots, m\}$.
         
         Note that for $r\in G$ and $b\in B_r$, 
     \[ R\big(\lambda_r^{\cl B}(b)\big) = \sum_{j=1}^m\theta_{c_j,d_j}(\lambda_r^{\cl B}(b))=\sum_{j=1}^m c_j\langle d_j, \lambda_r^{\cl B}(b)\rangle_{E_e}=0 \text{ if }r\not\in F.\]
    Therefore, for $g\not\in F$, $b\in B_g$,
    \begin{align*}
    &\|T_g^i(b)\|=\|E_g(\Phi_i(\lambda_g^{\cl B}(b)))^*E_g(\Phi_i(\lambda_g^{\cl B}(b)))\|^{1/2}\\&\leq \|E_e(\Phi_i(\lambda_g^{\cl B}(b))^*\Phi_i(\lambda_g^{\cl B}(b)))\|^{1/2}=\|\tilde\Phi_i(\lambda_g^{\cl B}(b))\|_{E_e}\\&=\|\tilde\Phi_i(\lambda_g^{\cl B}(b))-R(\lambda_g^{\cl B}(b))\|_{E_e}\leq\|\tilde\Phi_i-R\|\|\lambda_g^{\cl B}(b)\|_{E_e}<\varepsilon\|b\| \ , 
    \end{align*}
    showing that for each $i\in I$, $g\mapsto \|T_g^i\|$ is in $C_0(G)$.  
    Similarly, for each $b\in B_e$,
    \[\|T_e^i(b)\|\leq \|\tilde\Phi_i(\lambda_e^{\cl B}(b))\|_{E_e}\leq\|\tilde\Phi_i\|\|b\|,\]
    giving that $T^i$ is uniformly bounded (using b)). 
\end{proof}
Observe from the proof above that each $T_g^i$, $g\in G$, $i\in I$, is a $B_e$-bimodule map. 
When this happens we will say that $T^i=(T_g^i)_{g\in G}$ is a $B_e$-bimodule $\B$-bundle map. 
Note that if $T$ is a positive definite $\B$-bundle map, then $T$ is $B_e$-bimodule if and only if it is a right $B_e$-module map (as each $T_g$ then satisfies that $T_g(b)^*= T_{g^{-1}}(b^*)$ for all $g\in G$ and $b\in B_g$ cf.~\cite{bc25}). 

\begin{remark}
By combining Example \ref{DR-Hp} and Proposition \ref{moduleHaag}, one gets another proof of Proposition \ref{SigmaFell}.
\end{remark}

As a partial converse to Proposition \ref{moduleHaag}, we have the following result. We recall (cf.~\cite{landi_pavlov}) that an orthonormal basis for a right Hilbert module $X$ over a unital $C^*$-algebra $B$ is a family $\{u_j\}_{j\in J}$ such that $\langle u_j, u_k\rangle_B = 1_B$ if $j=k$ and $0$ otherwise and
for every $x\in X$ there exist $\widehat{x}_j \in B$ for all $j\in J$ such that \[
x = \sum_{j\in J} u_j\,\widehat{x}_j \quad \text{(unconditional convergence in norm)},\]
in which case $\widehat{x}_j= \langle u_j, x\rangle_B$ for every $j\in J$, cf.~\cite[Theorem 2.10]{landi_pavlov}, and $j\mapsto \|\widehat{x}_j\| \in C_0(J)$. 

\begin{theorem}\label{orthoH}
  Assume $\B=(B_g)_{g\in G}$ is a unital Fell bundle over $G$ 
   having the Haagerup-PD approximation property  with implementing net  $\{T^i\}_{i\in I}$ satisfying that 
    each $T_g^i$ is a compact, $B_e$-module map on $B_g$ (considered as a right Hilbert $B_e$-module).
  Moreover, assume that for every $g\in G$ there exists an orthonormal basis $\{u_\gamma(g): \gamma \in J_g\}$ for $B_g$.
  Then $C_r^*(\cl B)$ has the Hilbert $B_e$-module Haagerup property with respect to $E_e$.
\end{theorem}
\begin{proof}
    Since $B_e$ is unital, we may assume that $\sup_{i\in I} \|T_e^i\| \leq 1$. Let  $\{M_{T^i}\}_{i\in I}$ be the net of completely positive maps on $C_r^*(\cl B)$ associated to $\{T^i\}_{i\in I}$, cf.~Theorem \ref{PDCP}. As each $M_{T^i}$ is clearly a $B_e$-bimodule map, it is enough to show that the net $\{M_{T^i}\}_{i\in I}$ satisfies the conditions a)-d) of Definition \ref{def_module_HP}.  Since $\{T^i\}_{i\in I}$ is uniformly bounded, so is $\{M_{T^i}\}_{i\in I}$, cf. \cite[Theorem 3.14]{bc25}.  Fix $i\in I$.
    Then for $g\in G$ and $b_g\in B_g$ and finite $F\subset G$, using (\ref{M_T}), the right $B_e$-module property, and the fact that $T^i_e(1_{B_e})$ is a contractive self-adjoint central element of $B_e$, we get 
    \begin{align*}
        &(E_e\circ M_{T^i})\big(\big(\sum_{g\in F}\lambda_g^{\cl B}(b_g)\big)^*\big(\sum_{g\in F}\lambda_g^{\cl B}(b_g)\big)\big)=(E_e\circ M_{T^i})\big(\sum_{g,h\in F}\lambda_h^{\cl B}(b_h)^*\lambda_g^{\cl B}(b_g)\big)\\&=(E_e\circ M_{T^i})\big(\sum_{g,h\in F}\lambda_{h^{-1}g}^{\cl B}(b_h^*b_g)\big)=E_e\big(\sum_{g,h\in F}\lambda_{h^{-1}g}^{\cl B}(T_{h^{-1}g}^i(b_h^*b_g))\big)\\&=\sum_{g\in F}T_e^i(b_g^*b_g)=T_e^i(1_{B_e})\sum_{g\in F}b_g^*b_g
        \leq E_e\big(\big(\sum_{g\in F}\lambda_g^{\cl B}(b_g)\big)^*\big(\sum_{g\in F}\lambda_g^{\cl B}(b_g)\big)\big).
    \end{align*}
 Using density arguments we conclude that $E_e\circ M_{T^i}\leq E_e$. As $\{M_{T^i}\}_{i\in I}$ is uniformly bounded, so is the corresponding net $\{\widetilde{M}_{T^i}\}_{i\in I}$.
Also, for $g \in G$ and  $b\in B_g$, we have
 \begin{align*}
     \|\widetilde M_{T^i}(\lambda_g^{\cl B}(b))-\lambda_g^{\cl B}(b)\|_{E_e}=\|\lambda_g^{\cl B}(T_g^i(b)-b)\|_{E_e}\leq \|T_g^i(b)-b\|\to_i 0.
 \end{align*} 
 As  $\{\widetilde{M}_{T^i}\}_{i\in I}$ is uniformly bounded, it follows that $\|\widetilde M_{T^i}(x)-x\|_{E_e}\to_i 0$ for any $x\in \cl H_{\cl B}$. 

 Thus it remains only to show that each $\widetilde{M_{T^i}}$ is compact. Fix again  $i\in I$. As $g\mapsto \|T_g^i\|$ vanishes at infinity, for each $\varepsilon>0$, we can find a finite set $F(\varepsilon)\subset G$ such that $\|T_g^i\|<\varepsilon$ for any $g\not\in F(\varepsilon)$. 
 
 For every  $\mathcal{C}= (C_g)_{g\in F(\varepsilon)}$, where each $C_g$ is a finite subset of $J_g$ for every $g\in F(\varepsilon)$, we let $R_{\varepsilon, \mathcal{C}}:\mathcal{H}_\B\to\mathcal{H}_\B$ be the finite rank operator given by
\[R_{\varepsilon, \mathcal{C}}=\sum_{g\in F(\varepsilon)}\sum_{\gamma\in C_g}\theta_{\lambda_g^{\cl B}(T_g^i(u_\gamma(g))),\lambda_g^{\cl B}(u_\gamma(g))}.\]
It is quite obvious that the set $\mathcal{M}$ of elements of the form
\begin{equation} \label{xinM} x = \sum_{s\in S} \sum_{\gamma \in C'_s} \lambda_s^{\cl B}(u_{\gamma}(s)) \widehat{x}_\gamma(s),
\end{equation}
where  $S$ is a finite subset of $G$, $C'_s$ is a finite subset of $J_s$ for each $s\in S$,
and $\widehat{x}_\gamma(s) \in B_e$ for all $s\in S$ and $\gamma \in C'_s$, is  dense in $\mathcal{H}_\B$. Indeed, any element $w$ in $\mathcal{H}_\B$ can be approximated in $\|\cdot\|_{E_e}$-norm by some element  in $C_r^*(\B)$, which itself can be approximated in $\|\cdot\|_r$-norm, hence in $\|\cdot\|_{E_e}$-norm, with some element of the form 
$\sum_{s\in S} \lambda_s^\B(d_s)$, where $S$ is a finite subset of $G$ and $d_s\in B_s$ for every $s\in S$. Since $\{u_\gamma(s)\}_{\gamma\in J_s}$ is an orthonormal basis for $B_s$, we  have $$d_s= \sum_{\gamma\in J_s} u_\gamma(s) (u_\gamma(s)^* d_s),$$ so $d_s$ can be approximated in the norm of $B_s$ by $\sum_{\gamma\in C'_s} u_\gamma(s) (u_\gamma(s)^* d_s)$ for some finite subset $C'_s$ of $J_s$ for every $s\in S$. Since $\lambda_s^\B: B_s \to (C_r^*(\B), \|\cdot\|_r)$ is isometric, $\lambda_s^\B(d_s)$ can be approximated in $\|\cdot\|_r$-norm, hence in $\|\cdot\|_{E_e}$-norm, by $\sum_{\gamma\in C'_s} \lambda_s^\B(u_\gamma(s)) (u_\gamma(s)^* d_s)$ for every $s\in S$.  Altogether, this means that $w$ can be approximated in $\|\cdot\|_{E_e}$-norm by some element  of the form 
\[\sum_{s\in S} \sum_{\gamma \in C'_s} \lambda_s^{\cl B}(u_{\gamma}(s))(u_\gamma(s)^* d_s), \]
which belongs to $\mathcal{M}$.

Let $x$ be as in (\ref{xinM}). Note that 
\[ \|x\|_{E_e}^2 = \| E_e(x^*x)\| = \Big\| \sum_{s\in S,\gamma\in C'_s} \widehat{x}_\gamma(s)^*\widehat{x}_\gamma(s) \Big\|. \] Moreover, we have
 \begin{align*}
R_{\varepsilon, \mathcal{C}} (x)
&= \sum_{g\in F(\varepsilon),\, \gamma\in C_g,  s\in S, \gamma'\in C'_s}
\lambda_g^{\cl B}(T_g^i(u_\gamma(g)))\, \Big\langle \lambda_g^{\cl B}(u_\gamma(g)),  \lambda_s^{\cl B}(u_{\gamma'}(s)) \widehat{x}_{\gamma'}(s)\Big\rangle_{E_e}\\
&=\sum_{g\in F(\varepsilon)\cap S,\, \gamma\in C_g, \gamma'\in C'_g}
\lambda_g^{\cl B}(T_g^i(u_\gamma(g)))\, \Big\langle \lambda_g^{\cl B}(u_\gamma(g)),  \lambda_g^{\cl B}(u_{\gamma'}(g)) \widehat{x}_{\gamma'}(g)\Big\rangle_{E_e}\\
&=\sum_{g\in F(\varepsilon)\cap S,  \gamma\in C_g\cap C'_g} \lambda_g^{\cl B}(T_g^i(u_\gamma(g)))
\, \widehat{x}_{\gamma}(g).
\end{align*}

Now,  $T_g^i$ is a $B_e$-bimodule-map (since $T^i$ is positive definite) and $u_\gamma(g)u_\gamma(g)^* \in B_e$ for every $g\in G$ and $\gamma \in J_g$, so we get
\[ T_g^i(u_\gamma(g)) = T_g^i(u_\gamma(g)u_\gamma(g)^*u_\gamma(g))= u_\gamma(g)u_\gamma(g)^*T_g^i(u_\gamma(g)), \]
and, as $T_g^i(u_\gamma(g))=\sum_{\gamma'\in J_g}u_{\gamma'}(g)b^i_{\gamma'}(g)$ for some $b^i_{\gamma'}(g)\in B_e$, we further obtain
\[T_g^i(u_\gamma(g)) =u_\gamma(g)u_\gamma(g)^*T_g^i(u_\gamma(g))=u_\gamma(g)u_\gamma(g)^*\sum_{\gamma'\in J_g}u_{\gamma'}(g)b^i_{\gamma'}(g)=u_\gamma(g)b^i_\gamma(g),\]
giving $T_g^i(u_\gamma(g))=u_\gamma(g)b^i_\gamma(g)$. Thus, $\| T_g^i\| = \sup_{\gamma\in J_g} \|b^i_\gamma(g)\| $ for every  $g\in G$. So
$\|b^i_\gamma(g)\| < \varepsilon $ for $g\not \in F(\varepsilon)$ and $\gamma \in J_g$. 

Moreover,
\begin{align}
\label{T-R}
    \|\widetilde M_{T^i}(x)- R_{\varepsilon, \mathcal{C}}(x)\|_{E_e}^2
    &=\Big\|\sum_{g\in S\setminus(F(\varepsilon)\cap S), \, \gamma\in C'_g}\lambda_g^{\cl B}\Big(T_g^i(u_\gamma(g))\widehat{x}_\gamma(g) \Big)\nonumber\\&
    \quad \quad \quad \quad + \sum_{g\in F(\varepsilon)\cap S, \, \gamma\in C'_g\setminus C_g}\lambda_g^{\cl B}\Big(T_g^i(u_\gamma(g))\widehat{x}_\gamma(g) \Big)\Big\|_{E_e}^2\nonumber\\
    &=\Big\|\sum_{g\in S\setminus F(\varepsilon), \gamma, \gamma'\in C'_g}\widehat{x}_\gamma(g)^*T_g^i(u_\gamma(g))^*T_g^i(u_{\gamma'}(g))\widehat{x}_{\gamma'}(g)\\&+
    \sum_{g\in F(\varepsilon)\cap S, \gamma, \gamma'\in C'_g\setminus C_g}\widehat{x}_\gamma(g)^*T_g^i(u_\gamma(g))^*T_g^i(u_{\gamma'}(g))\widehat{x}_{\gamma'}(g)\Big\|\nonumber \\&
    =\Big\|\sum_{g\in S\setminus F(\varepsilon), \gamma, \gamma'\in C'_g}\widehat{x}_\gamma(g)^*b_\gamma^i(g)^*u_\gamma(g)^*u_{\gamma'}(g)b_{\gamma'}^i(g)\widehat{x}_{\gamma'}(g)\Big\|\nonumber\\&+
    \Big\|\sum_{g\in F(\varepsilon)\cap S, \gamma, \gamma'\in C'_g\setminus C_g}\widehat{x}_\gamma(g)^*b_\gamma^i(g)^*u_\gamma(g)^*u_{\gamma'}(g)b_{\gamma'}^i(g)\widehat{x}_{\gamma'}(g)\Big\|\nonumber \\&=
    \Big\|\sum_{g\in S\setminus F(\varepsilon), \gamma\in C'_g}\widehat{x}_\gamma(g)^*b_\gamma^i(g)^*b_{\gamma}^i(g)\widehat{x}_{\gamma}(g)\Big\|\nonumber\\&+
    \Big\|\sum_{g\in F(\varepsilon)\cap S, \gamma\in C'_g\setminus C_g}\widehat{x}_\gamma(g)^*b_\gamma^i(g)^*b_{\gamma}^i(g)\widehat{x}_{\gamma}(g)\Big\|\nonumber
\end{align}
Note  that if each $J_g$ is finite, then one can choose here $C_g=C'_g=J_g$ for every $g\in G$, and the last term above disappears. However, in the general case, we need to use the compactness assumption. Thus, each $T_g^i$ being compact, hence a limit of finite rank operators, and using that for any rank one operator $\theta_{\xi,\eta}$, $\xi$, $\eta\in B_g$, we have $\|\theta_{\xi,\eta}(u_\gamma(g))\|=\|\xi\widehat{\eta}_\gamma(g)^*\|\to_\gamma 0$, we readily obtain that $$\|b^i_\gamma(g)\|= \|u_\gamma(g)b^i_\gamma(g)\|=\|T_g^i(u_\gamma(g))\|\rightarrow_{\gamma} 0$$ for every $g\in G$. Thus, for each $g\in F(\varepsilon)$, we can find a finite subset $C_g$ of $J_g$ such that $\|b^i_\gamma(g)\| < \varepsilon$ for all $\gamma \in J_g \setminus C_g$. 
  
Using (\ref{T-R}), we get 
\begin{align*}
    \|\widetilde M_{T^i}(x)- R_{\varepsilon, \mathcal{C}}(x)\|_{E_e}^2
&=    \Big\|\sum_{g\in S\setminus F(\varepsilon), \gamma\in C'_g}\widehat{x}_\gamma(g)^*b_\gamma^i(g)^*b_{\gamma}^i(g)\widehat{x}_{\gamma}(g)\Big\|\\&+
    \Big\|\sum_{g\in F(\varepsilon)\cap S, \gamma\in C'_g\setminus C_g}\widehat{x}_\gamma(g)^*b_\gamma^i(g)^*b_{\gamma}^i(g)\widehat{x}_{\gamma}(g)\Big\|\\
    &\leq \Big(\sup_{g\not \in F(\varepsilon), \gamma \in C'_g} \|b_\gamma^i(g)\|^2\Big) \, \|x\|_{E_e}^2 + \Big(\sup_{g \in F(\varepsilon), \gamma \not \in C_g} \|b_\gamma^i(g)\|^2\Big) \, \|x\|_{E_e}^2  \\
    &\leq \varepsilon^2 \|x\|_{E_e}^2 + \varepsilon^2 \|x\|_{E_e}^2 = 2 \varepsilon^2 \|x\|_{E_e}^2.
    \end{align*}
 Thus we get $\|\widetilde M_{T^i}(x)- R_{\varepsilon, \mathcal{C}}(x)\|_{E_e} \leq \sqrt{2}\varepsilon \|x\|_{E_e}$ for all $x$ in $\mathcal{M}$. 
  As $\M$ is a dense subset of $\mathcal{H}_\B$ and $R_{\varepsilon, \mathcal{C}}$ is a finite rank map, it follows that $\widetilde M_{T^i}$ is compact. 
\end{proof}

\begin{corollary}\label{HaaPD-module}
 Let $\cl B=(B_g)_{g\in G}$ be a unital crossed product bundle over $G$,  with unitary section $u$ and associated twisted action $(\alpha, \sigma)$ of $G$ on $B_e$. 
    Assume $\B$ has the Haagerup-PD approximation property  with implementing net  $\{T^i\}_{i\in I}$ such that each $T^i$ is a right $B_e$-module map.
    Then $C_r^*(\cl B)$ has the Hilbert $B_e$-module Haagerup property with respect to $E_e$.
\end{corollary}
\begin{proof}
    For each $g\in G$, the one-point set $\{u(g)\}$ is an orthonormal basis for $B_g$. Thus the assumptions in Theorem \ref{orthoH} are satisfied.
\end{proof}
It is easy to see that Example \ref{DR-Hp} can be deduced from Corollary \ref{HaaPD-module}.  
We also mention an immediate consequence of this result.
\begin{corollary} 
 Let $\cl B=(B_g)_{g\in G}$ be a unital crossed product bundle over $G$ and assume that $G$ has the Haagerup property. Then $C_r^*(\cl B)$ has the Hilbert $B_e$-module Haagerup property with respect to $E_e$.   
\end{corollary}

\section{Fell bundles and nuclearity}
 We define below a notion of nuclearity for Fell bundles over discrete groups similar to the one introduced in \cite{mstt} for $C^*$-dynamical systems. (Note that a different definition is proposed in  \cite[Definition 6.1]{He}). 
\begin{definition}\label{nuclear_fellbundle}
    Let $G$ be a discrete group. We say that a Fell bundle $\mathcal B=(B_g)_{g\in G}$ is \emph{nuclear} 
    if there exists a net $\{T^i\}_{i\in I}$ of $\mathcal B$-bundle maps satisfying the following properties:
    \begin{enumerate}
        \item[a)] For each $i\in I$, $T^i=(T_g^i)_{g\in G}$ is positive definite and its support $\{g\in G: T_g^i\ne 0\}$ is finite;
        \item[b)] $\sup_{i\in I}\|T^i_e\|<\infty$;
        \item[c)] $T_g^i: B_g\to B_g$ is of finite rank for all $g\in G$ and all $i\in I$;
        \item[d)] $\lim_i\|T_g^i(b)-b\|=0$ for every $g\in G$ and $b\in B_g$. 
    \end{enumerate}
\end{definition}
In other words, $\B$ is nuclear when $\B$ has the PD-approximation property implemented by a net 
$\{T^i\}_{i\in I}$ of $\mathcal B$-bundle maps satisfying c).
If $\Sigma=(A, G, \alpha)$ is a discrete $C^*$-dynamical system, then it is straightforward to check that $\B_\Sigma$ is nuclear if and only if $\Sigma$ is nuclear in the sense of \cite{mstt}.
The equivalence of i) and iv) in the following result is known in the case of $C^*$-dynamical systems, cf.~\cite[Theorem 4.3]{mstt}. The equivalence of ii)--v) is known in general, cf.~\cite[Corollary 5.7]{bc25} (see also \cite[Theorems 3.20, 3.21]{bf25}). 
\begin{theorem}\label{nuclearC*}
Let $\mathcal B=(B_g)_{g\in G}$ be a Fell bundle. The following are equivalent:
\begin{enumerate}
    \item[i)] $\mathcal B$ is nuclear;
    \item[ii)] $B_e$ is nuclear and $\B$ has the PD-approximation property;
    \item[iii)] $B_e$ is nuclear and $\B$ is $C^*$-amenable;
    \item[iv)] $C_r^*(\mathcal B)$ is nuclear;
    \item [v)] $C^*(\B)$ is nuclear (where $C^*(\B)$ denotes the full cross-sectional $C^*$-algebra of $\B$ \cite{exel}). 
\end{enumerate}
\end{theorem}
\begin{proof} It suffices to show i) $\Rightarrow$ ii), and iv) $\Rightarrow$ i), but for the ease of the reader, we will also sketch a proof of i) $\Rightarrow$ iv).

i) $\Rightarrow$ ii) and iv): Let $\{T^i\}_{i\in I}$ be as in the definition of nuclearity of $\B$. Then $\mathcal B$ has the PD-approximation property and $\{T_e^i\}_{i\in I}$ is a net of completely positive finite rank maps on $B_e$ such that $T_e^i(b)\to_i b$ for any $b\in B_e$, showing that $B_e$ is nuclear. Thus ii) holds. 
 
Next we show that iv) also holds. We know from Theorem \ref{PDCP} that the corresponding map $M_{T^i}: C_r^*(\B)\to C_r^*(\B)$ is completely positive for each $i\in I$, and, using property b) of Definition \ref{nuclear_fellbundle}, we get  
\[\sup_{i\in I}\|M_{T^i}\|=\sup_{i\in I}\|T_e^i\|<\infty.\] 
Moreover, the properties a) and c) imply that each $T^i$ is of finite rank. Using property d) one readily checks that $\|M_{T^i}(x)-x\|\to 0$ for every $x$ in the canonical image $\iota^\B(C_c(\B))$ of $C_c(\B)$ in $C_r^*(\B)$. Since $\iota^\B(C_c(\B))$ is dense in $C_r^*(\B)$ and $\{M_{T^i}\}_{i\in I}$ is uniformly bounded, we deduce that $\|M_{T^i}(x)-x\|\to 0$ for every $x\in C_r^*(\B)$.This shows that $C_r^*(\B)$ is nuclear (see for example \cite[Theorem IV.3.1.5]{bla}). 

iv) $\Rightarrow$ i):
 Assume that $C_r^*(\mathcal B)$ is nuclear. Then there exist nets $\{\varphi_i\}_{i\in I}$  and $\{\psi_i\}_{i\in I}$ of completely positive contractive  maps $\varphi_i: C_r^*(\mathcal B)\to M_{k_i}(\mathbb C)$, $\psi_i: M_{k_i}(\mathbb C)\to C_r^*(\mathcal B)$ such that $\Phi_i=\psi_i\circ\varphi_i\to_i \text{id}$ pointwise in norm. As in the proof of \cite[Theorem 4.3]{mstt} we can arrange that the range of each $\psi_i$ lies in $\iota^\B(C_c(\B))$. For each $i\in I$ define now a $\B$-bundle map $T^i=(T^i_g)_{g\in G}$ by 
\[T_g^i(b)=E_g(\Phi_i(\lambda_g^{\mathcal B}(b)))\] for every $g\in G$ and $b\in B_g$. By \cite[Proposition 3.8]{bc25} each $T^i=(T_g^i)_{g\in G}$ is positive definite. 
Moreover, as both $E_e$ and $\Phi_i$ are contractive, we get that $\|T_e^i(b)\| \leq \|\lambda_e^{\mathcal B}(b))\| = \|b\|$ for all $b\in B_e$. 
Thus $\sup_{i\in I} \|T_e^i\| \leq 1$.

Fix $i\in I$. As the rank of $\Phi_i$ is finite and the range of $\Phi_i$ lies in $\iota^\B(C_c(\B))$, there exist $n(i)\in \mathbb{N}, g_k \in G$ and $b_k\in B_{g_k}$ for $k=1, \ldots, n(i)$
such that 
$$\text{Ran}\, \Phi_i \subset \text{Span}\{\lambda_{g_k}^{\mathcal B}(b_k): k\in \{1, \ldots, n(i)\}\}.$$
This implies that ${\rm Ran} \, T^i_g = {\rm Span} \{\lambda_{g_k}^\B(b_k): k\in\{1, \ldots, n(i)\}, g_k= g\}$ is finite dimensional,  and each $T^i$ is finitely supported with support contained in $\{g_k: k \in\{1, \ldots, n(i)\}\}$. 

Finally, letting $i$ vary, we get
$$\|T_g^i(b)-b\|=\|E_g(\Phi_i(\lambda_g^{\mathcal B}(b)))-E_g(\lambda_g^{\mathcal B}(b)))\|\leq\|\Phi_i(\lambda_g^{\mathcal B}(b))-\lambda_g^{\mathcal B}(b)\| \to_i 0$$
for all $g\in G$ and $b\in B_g$. This shows that all properties a)--d) in Definition \ref{nuclear_fellbundle} are satisfied, i.e., i) holds.
\end{proof}

\begin{corollary} \label{nuclear2} Let $\cl B=(B_g)_{g\in G}$ be a unital crossed product bundle over $G$,  with unitary section $u$ and associated twisted action $(\alpha, \sigma)$ of $G$ on $B_e$. 
    Assume $\B$ is nuclear and there exists a faithful $\alpha$-invariant tracial state $\tau$ on $B_e$. Then $(\B, \tau)$ has the Haagerup property. 
    \end{corollary}
\begin{proof} By Theorem \ref{HP-Fell} it suffices to show that $(C_r^*(\B), \tau\circ E_e)$ has the Haagerup property. Note that $\tau\circ E_e$ is a faithful tracial state on  $C_r^*(\B)$ (cf.~\cite[Corollary 2.3]{b91}). Further, $C_r^*(\B)$ is nuclear by Theorem \ref{nuclearC*}. Hence, by Suzuki's result \cite[Theorem 3.6]{suzuki13}, we  obtain that $(C_r^*(\B), \tau\circ E_e)$ has the Haagerup property.  
\end{proof}

\vspace{-2ex}
In general, if $\B$ is a nuclear unital Fell bundle over $G$ and $\psi$ is a state on $B_e$, it is an open question whether  $(\B, \psi)$ necessarily has the Haagerup property. 
For an example where the answer to this question is positive, let $\Sigma=(A, G, \alpha)$ and $\psi_\lambda$ be as in Example \ref{CAR-2}. Then $C_r^*(\Sigma)$ is nuclear (since $A$ and $G$ are abelian), so $\B_\Sigma$ is nuclear (by Theorem \ref{nuclearC*}). Moreover,  $(\B_\Sigma, \psi_\lambda)$ has the Haagerup property, as explained in Example \ref{CAR-2}. (Note that  $\psi_\lambda$ is not $\alpha$-invariant, so Corollary \ref{nuclear2} can not be applied to deduce that $(\B_\Sigma, \psi_\lambda)$ has the Haagerup property.)    
The answer to the question would be positive in general if one could generalize Suzuki's result \cite[Theorem 3.6]{suzuki13} and show that $(B, \varphi)$ has the Haagerup property whenever $B$ is a unital nuclear $C^*$-algebra and $\varphi$ is a state on $B$. 
\section{A final overview}\label{diagram}

Let $\B=(B_g)_{g\in G}$ be a Fell bundle over the discrete group $G$,  $\psi$ be a state on the unit fiber $B_e$, and $E_e: C_r^*(\B)\to B_e$ be the canonical conditional expectation.  
In this paper, we have considered the following properties:

\begin{itemize}
\item[(o)] $G$ is amenable;
\item[(i)] $G$ has the Haagerup property;
\item[(ii)] $(B_e,\psi)$ has the Haagerup property; 
\item[(iii)] $(C^*_r(\B),\psi \circ E_e)$ has the Haagerup property;
\item[(iv)] $\B$ has the PD-approximation property;
\item[(v)] $(\B,\psi)$ has the Haagerup property; 
\item[(vi)] $\B$ has the Haagerup PD-approximation property; 
\item [(vii)] $C_r^*(\B)$ has the Hilbert $B_e$-module Haagerup property w.r.t.~$B_e$; 
\item[(viii)] $\B$ is nuclear.
\end{itemize}

The mutual relationships between all these properties are shown in the diagram below.
 
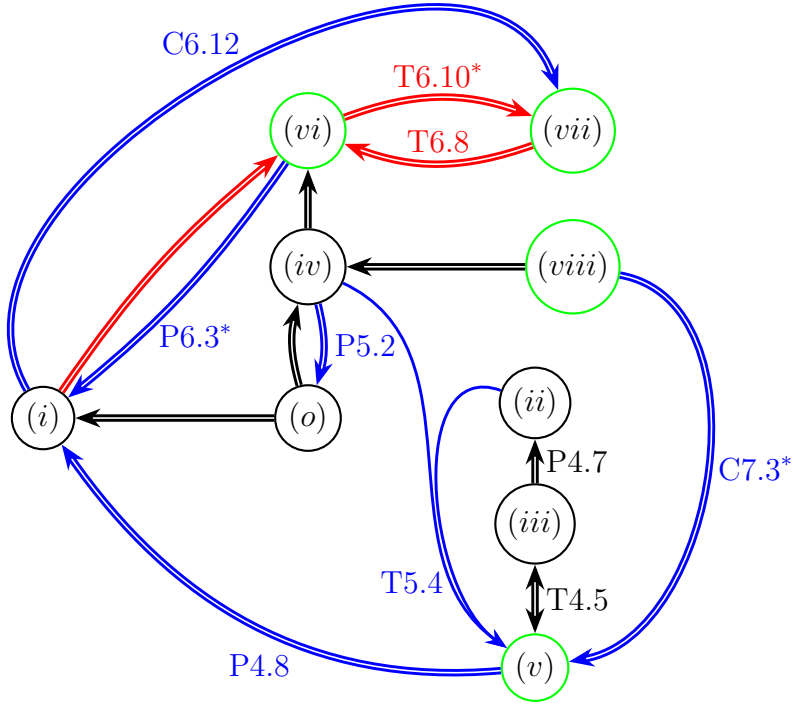
\vspace{-5ex}
\begin{tikzpicture}[
    >=Stealth,
    every node/.style={font=\large},
    line width=1.1pt
]

\node[circle, draw=green, thick, inner sep=2pt] (vi) at (0,5.8)  {$(vi)$};
\node[circle, draw=green, thick, inner sep=2pt] (vii) at (3.5,5.8) {$(vii)$};

\node[circle, draw=black, thick, inner sep=2pt] (iv) at (0,4.0) {$(iv)$};
\node[circle, draw=green, thick, inner sep=2pt] (viii) at (3.5,4.0)  {$(viii)$};

\node[circle, draw=black, thick, inner sep=2pt] (i) at (-3.5,2.0)  {$(i)$};
\node[circle, draw=black, thick, inner sep=2pt] (zero) at (0,2.0)  {$(o)$};

\node[circle, draw=black, thick, inner sep=2pt] (ii) at (3.0,2.2)  {$(ii)$};
\node[circle, draw=black, thick, inner sep=2pt] (iii) at (3.0,0.6)  {$(iii)$};
\node[circle, draw=green, thick, inner sep=2pt] (v) at (3.0,-1.3)  {$(v)$};

\draw[double, red, ->] (vi) to[bend left=20] node[above, red] {$\mathrm{T6.10}^*$} (vii);
\draw[double, red, ->] (vii) to[bend left=20] node[above, red] {$\mathrm{T6.8}$} (vi);

\draw[double, blue, ->]
    (vi) to[bend left=8]
    node[pos=.72, right=1mm, blue] {$\mathrm{P6.3}^*$}
    (i);

\draw[double, red, ->]
    (i) to[bend left=8]
    node[pos=.52, right, blue] {}
    (vi);

\draw[double, black,  ->]
    (iv) -- (vi);

\draw[double, black,  ->]
    (viii) -- (iv);
    
\draw[double, blue, ->]
    (iv) to[bend left=15] node[right, blue] {$\mathrm{P5.2}$} (zero);

\draw[double, black, ->]
    (zero) to[bend left=15] node[right, blue] {} (iv);

\draw[double, black, ->]
    (zero) -- (i);

\draw[blue, ->]
    (iv)
    to[out=-25,in=145]
    node[pos=.80, left] {$\mathrm{T5.4}$}
    (v);

\draw[blue, ->]
     (ii)
     to[out=-200,in=145] 
     (v);

\draw[double, black, ->] 
    (iii) -- node[right]{$\mathrm{P4.7}$}(ii);

\draw[double, black, <->]
    (iii) -- node[right] {$\mathrm{T4.5}$} (v);

\draw[double, blue, ->]
    (viii)
    to[out=-10,in=10]
    node[right, blue] {$\mathrm{C7.3}^*$}
    (v);

\draw[double, blue, ->]
    (v)
    to[out=185,in=-55]
    node[pos=.48, below=2mm, blue] {$\mathrm{P4.8}$}
    (i);

\draw[double, blue, ->]
    (i)
    to[out=-240,in=-250]
    node[pos=.48, above=3mm, blue] {$\mathrm{C6.12}$}
    (vii);

\end{tikzpicture}

While the properties (o)--(iv) should be familiar to many readers, the properties (v)--(viii), that are marked in green in the diagram above, have not been introduced before. 
The implication-arrows that are not decorated are the obvious ones. The label of those that are decorated indicates where the result can be found in the present article (T meaning Theorem, etc). The color of each arrow indicates the standing assumptions according to the following scheme:

\medskip 

black: $\B$ general Fell bundle

red: \,  $\B$ unital Fell bundle

blue: $\B$ unital crossed product bundle + $\psi$ invariant (except in C6.12) 

\medskip \noindent with some additional assumptions in P6.3, T6.10 and C7.3 (therefore marked with a *).

\bigskip 
{\bf Acknowledgment}: This article is partly based on research supported by the Swedish Research Council under grant no.~2021-06594, done while the authors were in residence at Institut Mittag-Leffler in Djursholm during the semester "Operator algebras and Quantum information theory", Spring 2026.
L.T.~was supported by the Swedish Research Council project grant 2023-04555. E.B.~and L.T.~thank Sapienza University of Rome and R.C.~for the kind hospitality during their stay in  August 2026, when the first draft of the present paper was finalized.


\begin{thebibliography}{99} 
\bibitem[AF19]{af19} {\sc F.~Abadie, D.~Ferraro}: {\it Equivalence of Fell bundles over groups}, {\rm J. Operator Theory} {\bf 81} (2019), 273--319.

\bibitem[ABF22]{abf}
{\sc F.~Abadie, A.~Buss, D.~Ferraro},
{\it Amenability and approximation properties for partial actions and Fell bundles},
{\rm Bull. Braz. Math. Soc. (N.S.)} {\bf 53} (2022), 173–227.

\bibitem[AD87]{Claire}
{\sc C.~Anantharaman-Delaroche}, {\it Syst\'emes dynamiques non commutatifs et moyennabilit\'e}, {\rm Math. Ann.} {\bf 279} (1987), 297--315.

\bibitem[ADPo]{adpo}
{\sc C.~Anantharaman-Delaroche, S. Popa}, {\it An introduction to type $II_1$ factors}, book preprint
{\url https://www.math.ucla.edu/~popa/Books/IIunV15.pdf}.

\bibitem[B91]{b91}
{\sc E.~B\'edos}, {\it Discrete groups and simple $C^*$-algebras}, 
{\rm Math. Proc. Cambridge Philos. Soc.}, {\bf 109} (1991), 521--537.

\bibitem[BC12]{bc12}
{\sc E.~B\'edos, R.~Conti}, {\it On discrete twisted $C^*$-dynamical systems, Hilbert $C^*$-modules and regularity}, 
{\rm M\"unster J. Math.} {\bf 5} (2012), 183--208.

\bibitem[BC15]{bc15}
{\sc E.~B\'edos, R.~Conti}, {\it
Fourier series and twisted C*-crossed products}, {\rm J. Fourier Anal. Appl.} {\bf 21} (2015), 32--75.

\bibitem[BC16]{bc16}
{\sc E.~B\'edos, R.~Conti}, {\it The Fourier-Stieltjes algebra of a {$C^*$}-dynamical system}, {\rm Internat. J. Math.} {\bf 27} (2016), 1650050, 50 pp. 

\bibitem[BC25]{bc25}
{\sc E.~B\'edos, R.~Conti},
{\it Positive definiteness and Fell bundles over discrete groups}, {\rm Math.~Z.} {\bf 311} (2025), no.~1, Paper No.~9, 26 pp.

\bibitem[Bla06]{bla} {\sc B.E.~Blackadar}, {\it Operator algebras}, Encyclopaedia of Mathematical Sciences Operator Algebras and Non-commutative Geometry, 122 III, Springer, Berlin, 2006.

\bibitem[BO08]{bo}
{\sc N.P.~Brown, N.~Ozawa}, {\it {$C^*$}-algebras and finite dimensional approximations}, {\rm American Mathematical Society}, 2008.

\bibitem[BF25]{bf25}
{\sc A.~Buss, D.~Ferraro},
{\it Characterizations of amenability for noncommutative dynamical systems and Fell bundles},
arXiv:2510.15581 

\bibitem[BF25b]{bf25b}
{\sc A.~Buss, D.~Ferraro},
{\it $W^*$-amenability for Fell bundles over discrete groups},
arXiv:2512.16524 

\bibitem[BK26]{bk26}
{\sc A.~Buss, P.~Karmakar},
{\it Rapid decay and localizability for Fell bundles over \`etale groupoids},
arXiv:2604.23907.

\bibitem[BEW24]{bew24}
\sc{ A.\,Buss, S.\, Echterhoff, R.\, Willett},
{\it Amenability and weak containment for actions of locally compact groups on $C^*$-algebras}, 
{\rm Mem.~Amer.~Math.~Soc.} {\bf 301} (2024), no. 1513, v+88 pp. 


\bibitem[BKMS24]{bkms}
{\sc A.~Buss, B.~Kwaśniewski, A.~McKee, A.~Skalski},
{\it Fourier--Stieltjes category for twisted groupoid actions}, {\rm preprint}, arXiv:2405.15653.

\bibitem[CCJJV01]{ccjjv} {\sc P.A.~Cherix, M.~Cowling, P.~Jolissaint, P.~Julg, A.Valette}, {\it Groups with the Haagerup property. Gromov's a-T-menability}, Progress in Mathematics, {\bf 197}. Birkh\"auser Verlag, Basel, 2001.

\bibitem[Cho83]{choda}
{\sc M.~Choda}, {\it Group factors of the Haagerup type}, {\rm Proc.~Japan Acad.} {\bf 59} (1983), {174--177}.

\bibitem[Dong10]{dong}
{\sc Z.~Dong}, {\it Haagerup property for {$C^*$}-algebras}, {\rm J.~Math.~Anal.~Appl.} {\bf 377} (2010), {631--644}.

\bibitem[DoR12]{dong_ruan}
{\sc Z.~Dong, Z.J.~Ruan}, {\it A Hilbert module approach to the Haagerup property}, {\rm Integr. Equ. Oper. Theory}, {\bf 73}, (2012), {431--454}.

\bibitem[Ex97a]{exel97a} {\sc R.\,Exel}, {\it Twisted partial actions: a classification of regular $C^*$-algebraic bundles}, {\rm Proc.~London Math. Soc.}{\bf 74} (1997), 417-–443.

\bibitem[Ex97b]{exel97} {\sc R.~Exel}, {\it Amenability for Fell bundles}, {\rm J.~Reine Angew.~Math.\/}  {\bf 492} (1997), 41--73.

\bibitem[Ex17]{exel} {\sc R.\,Exel}, {\it Partial dynamical systems, {F}ell bundles and applications.
    Mathematical Surveys and Monographs}, {\rm {\bf 224}, American Mathematical Society, Providence, RI, 2017}.

\bibitem[Fl25]{fl25} {\sc F.~Flores}, {\it Polynomial growth and functional calculus in algebras of integrable cross-sections},
{\rm J. Math. Anal. Appl.} {\bf 549} (2025), no. 2, Paper No. 129486, 32 pp.

\bibitem[GM22]{gm22} {\sc C.\,Gao, Q.\,Meng}, {\it On the Haagerup property of $C^*$-dynamical systems}, {\rm J.~Math.~Research with Appl.} {\bf 42} (2022), 628-636.

\bibitem[Haa78-79a]{haa1} 
{\sc U.~Haagerup}, {\it  On the dual weights for crossed products of von Neumann algebras. II. Application of operator-valued weights},
{\rm J. Math. Scand.} {\bf 43} (1978/79), 119--140.

\bibitem[Haa78-79b]{haa2}
{\sc U.~Haagerup}, {\it  An example of a nonnuclear $C^*$-algebra, which has the metric approximation property },
{\rm Invent. Math.} {\bf 50} (1978/79), 279--293.

\bibitem[He21]{He} {\sc W.~He}, {\it Herz-Schur multipliers of Fell bundles and the nuclearity of the full $C^*$-algebra}, {\rm Int.~J.~Theoretical and Applied Math.} {\bf 7} (2021), 17--29.

\bibitem[Ho11]{hou}
  {\sc C. Houdayer}, {\it An introduction to $II_1$ factors}, 2011, 

https://api.semanticscholar.org/CorpusID:133595270. 

\bibitem[Jol02]{jolisaint}
{\sc P.~Jolissaint}, {\it Haagerup approximation property for finite von {N}eumann algebras}, {\rm Journal of Operator Theory}, {\bf 48} (2002), {539--551}.

\bibitem[KLS22]{kls}{\sc B.~Kwasniewski, K.~Li, A.~Skalski}, {\it The Haagerup property for twisted groupoid dynamical systems}, {\rm J.~Funct.~Anal.} {\bf 283} (2022), no. 1, Paper No. 109484, 43 pp.

\bibitem[LP12]{landi_pavlov}{\sc G. Landi, A. Pavlov}, {\it On orthogonal systems in Hilbert $C^*$-modules}, {\rm J. Operator Theory} {\bf 68} (2012), no. 2, 487–500.

\bibitem[LF16]{lf16} {\sc C.~Li, X.~Fang}, {\it The (**)-Haagerup property for $C^*$-algebras}, {\rm Chin.~Ann.~Math.}, {\bf 37B} (2016), 367--372.

\bibitem[MaWa25]{mawa}{\sc N.~Machado, S.~Wagner}, {\it Saturated Fell bundles and their classification}, {\rm arXiv:2501.14472}. 

\bibitem[MTT18]{mtt}{\sc A.~McKee, I.G.~Todorov, L.~Turowska}, {\it Herz-Schur multipliers of dynamical systems}, {\rm Adv.~Math.}{\bf 331} (2018), 387--438.

\bibitem[MPTT22]{mptt}{\sc A.~McKee, R.~Pourshahami, I.G.~Todorov,  L.~Turowska}, {\it Central and convolution Herz-Schur multipliers}, {\rm New York J.~Math.} {\bf 28} (2022), 1--43.

\bibitem[MT21]{mt}{\sc A.~McKee, L.~Turowska}, {\it Exactness and SOAP of crossed products via Herz-Schur multipliers},  {\rm J.~Math.~Anal.~Appl.} {\bf 496} (2021), no. 2, Paper No. 124812, 16 pp.

\bibitem[MSTT18]{mstt}{\sc A.~McKee, A.~Skalski, I.G.~Todorov, L.~Turowska}, {\it Positive Herz-Schur multipliers and approximation properties of crossed products}, 
{\rm Math.~Proc.~Cambridge Philos.~Soc.} {\bf 165} (2018), 511--532.

\bibitem[MW21]{mw21} {\sc Q.~Meng, L.~Wang}, {\it The Haagerup approximation property for $C^*$-algebras}, {\rm Linear and multilinear algebra} {\bf 69} (2021), 1275--1285.

\bibitem[PR89]{pr} {\sc J.A.~Packer, I.~Raeburn}, {\it Twisted crossed product of $C^*$-algebras}, {\rm Math.~Proc.~Camb.~Phil.~Soc.} {\bf 106} (1989), 293--311.

\bibitem[P67]{powers} {\sc R.T.~Powers}, {\it Representations of uniformly hyperfinite algebras and their associated von Neumann rings}, {\rm Ann.~of Math.} {\bf 86} (1967), 138--171.

\bibitem[Suz13]{suzuki13} {\sc Y.~Suzuki}, {\it Haagerup property for $C^*$-algebras and rigidity of $C^*$-algebras with property (T)}, {\rm J.~Funct.~Anal.} {\bf 265} (2013), 1778--1799.

\bibitem[Ta14]{takeishi}{\sc T.~Takeishi}, {\it On nuclearity of $C^*$-algebras of Fell bundles over \'etale groupoids}, {\rm Publ. Res. Inst. Math. Sci.} {\bf 50} (2014), no. 2, 251–268.

\bibitem[ZM68]{zm} {\sc G.~Zeller-Meier}, {\it Produits crois\'{e}s d'une $C^*$-alg\`{e}bre par un groupe d'automorphismes}, {\rm  J. Math.~Pures Appl.} {\bf 47} (1968), 101--239.

\end{thebibliography}
\end{document}